\documentclass[11pt]{article}
\usepackage{amsfonts,amsthm}
\usepackage{amsmath,amssymb,enumerate,color}
\usepackage{algorithm,algorithmic}
\usepackage{epsfig,graphicx,psfrag,mathrsfs}
\usepackage{subfig}
\usepackage{float}
\usepackage{eucal}
\usepackage{yhmath}
\usepackage{hyperref}
\hypersetup{colorlinks}
\RequirePackage{multirow}
\usepackage{booktabs}
\usepackage{array} 

\usepackage{multirow}
\usepackage{textcomp}
\usepackage[]{fontenc}
\usepackage{extarrows}

\usepackage{rotating}
\usepackage{lscape}
\usepackage{bm}
\usepackage{xspace}
\usepackage{tikz}
\usepackage{amsbsy}
\usepackage{geometry}
\usepackage{tabularx}

\DeclareMathOperator*{\argmin}{argmin}

\DeclareMathOperator{\new}{new}
\def\Kzero{\mathcal{K}_0}
\def\Kone{\mathcal{K}_1}

\DeclareMathOperator{\T}{T}

\def\bbR{\mathbb{R}}

\def\xnew{x_{\text{new} }}
\def\new{\text{new}}
\def\fhat{\hat{f}}
\def\Ghat{\hat{G}}
\def\ahat{\hat{\alpha}}
\def\xhat{\hat{x}}
\def\shat{\hat{s}}
\def\ghat{\hat{g}}
\def\Konehat{\hat{\mathsf{K}}_1}

\newtheorem{proposition}{Proposition}[section]
\newtheorem{theorem}{Theorem}[section]
\newtheorem{lemma}{Lemma}[section]

\newtheorem{assumption}{Assumption}[section]

\theoremstyle{definition}
\newtheorem{definition}{Definition}[section]
\newtheorem{remark}{Remark}[section]

\numberwithin{equation}{section}
\numberwithin{figure}{section}
\numberwithin{table}{section}

\graphicspath{{figures/}{figure/}{pictures/}%
	{picture/}{pic/}{pics/}{image/}{images/}}

\title{Kahan's Automatic Step-Size Control for Unconstrained Optimization}

\author{Yifeng Meng\thanks{School of Mathematical Sciences,
Fudan University, Shanghai, 200433, China. Email: {\tt  yfmeng23@m.fudan.edu.cn}}
\and Chungen Shen\thanks{College of Science, University of Shanghai for Science and Technology, Shanghai, 200093, China. Email: {\tt shenchungen@gmail.com}}
\and Linuo Xue\thanks{Institute of Science and Technology for Brain-Inspired Intelligence, Fudan University, Shanghai, 200433,  China. Email: {\tt  lnxue23@m.fudan.edu.cn}}
 \and Lei-Hong Zhang\thanks{Corresponding author. School of Mathematical Sciences, Soochow University, Suzhou, 215006, Jiangsu, China.
        This work was
 supported in part by the National Natural Science Foundation of China (NSFC-12471356,  NSFC-12371380), Jiangsu Shuangchuang Project (JSSCTD202209), Academic Degree and Postgraduate Education Reform Project of Jiangsu Province, and China Association of Higher Education under grant 23SX0403. 
        Email: {\tt longzlh@suda.edu.cn}.}
}

\begin{document}
\date{ }
\maketitle

\begin{abstract}
The Barzilai and Borwein (BB) {gradient method} is one of the most widely-used line-search gradient methods. It computes the step-size for the current iterate by using the information carried in the previous iteration. Recently, William Kahan  [{\it Kahan,  Automatic Step-Size Control for Minimization Iterations, Technical report, University of California, Berkeley CA, USA, 2019}] proposed new Gradient Descent (KGD) step-size strategies which iterate the step-size itself by effectively utilizing the information in the previous iteration. In the quadratic model, such a new step-size is shown to be mathematically equivalent to the long BB step, but no rigorous mathematical proof of its efficiency and effectiveness for the general unconstrained minimization is available.  In this paper, by this equivalence with the long BB step, we first derive a short version of KGD step-size and show that, for the strongly {convex} quadratic model with a Hessian matrix $H$, both the long and short KGD  step-size (and hence BB step-sizes)  {gradient methods} converge at least R-linearly with a rate $1-\frac{1}{{\rm cond}(H)}$.  For the general unconstrained minimization, we further propose an adaptive framework to effectively use the KGD step-sizes; global convergence and local R-linear convergence rate are proved. Numerical experiments are conducted on the CUTEst collection {as well as the practical logistic regression problems,  and  we  compare the performance  of the proposed methods with various BB step-size approaches and other recently proposed  adaptive gradient methods to demonstrate the efficiency and robustness.}
\end{abstract}

\medskip
{\small
\medskip

}

\section{Introduction}
Given a continuously differentiable function $f: \bbR^n \to \bbR$, we consider the unconstrained minimization
    \begin{equation}\label{eq:prob1}
        \min_{x \in \bbR^n } f(x).
    \end{equation}
The classical line-search gradient method is to solve the problem  \eqref{eq:prob1} by iterating the approximation $x_k$ according to $x_{k+1} = x_k - \alpha_k g_k$ until convergence, where $g_k = \nabla f(x_k)$ is the gradient at $x_k$ and $\alpha_k \geq 0$ is a properly chosen step-size.
    \par
    A natural choice of the step-size is the so-called  Cauchy step (see e.g., \cite{nowr:2006}), which is computed by solving exactly the step-size $\alpha_k^*$ defined by
    \begin{equation}\label{eq:cauchy}
        \alpha_k^* = \argmin_{\alpha \geq 0} f(x_k - \alpha g_k).
    \end{equation}
    It is known that for a strongly convex quadratic function (i.e., $H\in \bbR^{n\times n }$ is symmetric and positive definite)
    \begin{equation}\label{eq:quadf}
        f(x) = \frac{1}{2} x^{\T} H x + b^{\T} x,
    \end{equation}
    the sequence $\{x_k\}$ converges Q-linearly to $x^*=-H^{-1}b$ with the linear rate \cite{nowr:2006}
    %\begin{equation}\nonumber\label{eq:linearfactor}
       $ \varrho = \frac{\kappa - 1 }{\kappa +1 }$ 
    %\end{equation}
    where $\kappa=\frac{\lambda_{n}(H)}{\lambda_{1}(H)}\geq 1$ is the condition number of the Hessian matrix $H$ and $\lambda_n(H)\ge \lambda_{n-1}(H)\ge \dots\lambda_1(H)>0$ are the  ordered eigenvalues of $H$.  In practice, however, computing the Cauchy step \eqref{eq:cauchy} accurately can be inefficient, and more importantly, the resulting sequence $\{ x_k \}$ will still suffer from  slow convergence and zigzag behavior (see e.g., \cite{nowr:2006}). This is the case even for the quadratic model \eqref{eq:quadf}. As one of the most efficient and widely-used practical approaches for the step-size, in 1988, Barzilai and Borwein  \cite{babo:1988} introduced the long and short choices, commonly referred to as BB steps, $\alpha_k^{\text{BB1}} $ and $\alpha_k^{\text{BB2}} $, respectively,  through  minimizing the residuals of certain secant equations in $\ell_2$ norm: 
    \begin{equation}{\label{BBl} }
        \alpha_k^{\text{BB1}} = \argmin_{\alpha \in \bbR} \|\alpha^{-1} s_{k-1} - y_{k-1}\|_2 = \frac{s_{k-1}^{\T} s_{k-1} }{s_{k-1}^{\T} y_{k-1} }
    \end{equation}
    and
    \begin{equation}{\label{BBs}}
        \alpha_k^{\text{BB2}} = \argmin_{\alpha \in \bbR} \|s_{k-1} -\alpha y_{k-1}\|_2 = \frac{s_{k-1}^{\T} y_{k-1} }{y_{k-1}^{\T} y_{k-1} },
    \end{equation}
    where $s_{k-1} = x_k - x_{k-1}$ and $y_{k-1} = g_k - g_{k-1}$. Besides  favorable numerical performances in practical applications, there are many works (see, e.g., \cite{budh:2019,dafl:2005,dali:2002,flet:2005,lisu:2021,rayd:1993,rayd:1997}) on the convergence in theory as well as its extensions since 1988. For example, in case of the  strongly convex quadratic model \eqref{eq:quadf}, {(i) \cite{rayd:1993} proves the global   convergence of both the long and short BB methods, (ii) \cite{dali:2002} contributes  the R-linear convergence of the long BB step \eqref{BBl} method, and (iii) a recent development given in \cite{lisu:2021} further shows that  the sequence $\{x_k\}$  converges to the solution $-H^{-1}b$ at least linearly with a rate of $1- \frac{1}{\kappa}$.} However, for a general  strongly convex function $f(x)$, {the  gradient method with the long BB step} \eqref{BBl} may not be convergent and counter-example of divergence with cycles between four points is given in \cite{budh:2019}; to ensure  global convergence, adaptive strategies of effectively using  $\alpha_k^{\text{BB1}} $ and/or $\alpha_k^{\text{BB2}} $  in a globalization framework have been discussed in \cite{budh:2019}, and favorable numerical performance is reported on a test of CUTEst collection problems \cite{gort:2015}.
        
    %We will discuss some recent developments in Section \ref{sec:quad}.

    Another recent development for accelerating the line-search gradient methods is made by the numerical analyst William Kahan \cite{kahn:2019b,kahn:2019a}.  One of his ideas is to distinguish two scenarios for the current iterate $x_k$ and then propose different automatic step-size control for $\alpha$. In particular, he defined two regimes associated with the current iterate $x_k$ as follows:
    \begin{equation}\label{eq:Regime01}
	\hspace{5mm}
	\framebox{\parbox{13cm}{
			%\begin{itemize}
			{\bf {\bf Regime 0} \cite{kahn:2019b}:} $x_k$ is rather farther from a sought minimum $x_*$ than from other stationary points. Hessian $H(x_k)=\nabla^2 f(x_k)$ varies enough to thwart attempts to infer good values for hyper-parameters other than step-sizes.\\\\
			{\bf {\bf Regime 1} \cite{kahn:2019b}:} $x_k$ is much closer to a minimum than to all other stationary points, and the largest and least eigenvalues of $H(x_k)$ vary slowly. 
			}
		}
	\end{equation}
    Whenever the scenario where the current $x_k$ is situated can be detected, automatic step-size strategies for Regime 0 and Regime 1, which we call Kahan's gradient descent (KGD) step-sizes to be presented in Section \ref{sec:KGD}, can be employed without using explicitly the condition number   and the norm of the convergent Hessian matrix. Heuristic rules for determining the two regimes in \eqref{eq:Regime01} are presented. However, apart from some preliminary  numerical experiments claimed in \cite{kahn:2019b,kahn:2019a}, there is no systematic framework in efficiently globalizing the KGD step-sizes in two scenarios, and moreover, no rigorous mathematical proof of the efficiency and effectiveness of KGD step-size strategies is available. 
    \par
    A recent  result \cite{gjmx:2023} relevant to Kahan's work in \cite{kahn:2019b,kahn:2019a}  has connected the KGD step-size strategy for Regime 1 with the long BB step-size \eqref{BBl} when the underlying objective function $f(x)$ is  strongly convex quadratic. This connection, together with some recent developments \cite{budh:2019,dafl:2005,dahs:2006,lisu:2021}  on the BB steps, makes it possible to rigorously provide a mathematical proof for the convergence of the {KGD step-size method} for Regime 1 for general convex problems; moreover, inspired also by this connection, we propose a counterpart of KGD step-size which reduces to the short BB step-size \eqref{BBs} in the strongly convex quadratic case; the newly proposed KGD step-size for  Regime 1 is then named as the short KGD step-size,  in contrast the long KGD step-size originally presented in \cite{kahn:2019b,kahn:2019a}. A more clear description of our contributions in this paper is as follows.
    \begin{itemize}
    \item[(1)] By a relation between the long and short BB step-sizes in different variable spaces, and the recent convergence analysis \cite{lisu:2021} for the long  BB step-size, we prove that the short {BB step-size method} also converges to the solution  at least linearly with a rate of $1- \frac{1}{\kappa}$ for the strongly convex quadratic model \eqref{eq:quadf}; see Theorem \ref{thm3.2}.
    
    \item[(2)] We propose a short KGD step-size strategy for Regime 1, and show that both long and short KGD step-size {{methods}} converge to the solution  at least linearly with a rate of $1- \frac{1}{\kappa}$ for the strongly convex quadratic model \eqref{eq:quadf}; see Theorems  \ref{thm3.1} and \ref{thm3.2}. 
    
    \item[(3)] For the strongly convex problem, we show that both long and short KGD step-size {methods} converge at least R-linearly; see Theorems  \ref{thm:global} and \ref{localcon}.

    \item[(4)] For the general unconstrained minimization \eqref{eq:prob1}, we propose an effective adaptive KGD step-size framework which dynamically treats the two regimes \eqref{eq:Regime01} and converges globally; see Algorithm \ref{alg:global} and Theorem \ref{thm:global}.
    
    \item[(5)] We conduct numerical tests on the CUTEst collection \cite{gort:2015} {as well as the practical logistic regression problems \cite{mami:2020},  and  compare the performance  of the proposed methods with various BB step-size approaches and other recently proposed adaptive gradient methods \cite{mami:2020} to demonstrate the efficiency and robustness}; see Section \ref{sec:numer}.

    \end{itemize}

    %During our study, we surprisingly find a connection between Kahan's automatic step and the long BB step in \eqref{BBl}. This fact inspires our proof in Section \ref{sec:globalalgorithm} and leads to an unexpacted step, called \textbf{Short KGD Step} in Section \ref{sec:quad}. Combining with line-search technique and steps in \textbf{Regime 0}, we give the global Kahan's gradient method. Further more, we test our algorithm with unconstrained optimization problem from CUTEst collection. In numerical experiments, our method performs more stable and efficient than the classic algorithms, which is an encouraging result. 
    \par
    The rest of the paper is organized as follows. In Section \ref{sec:KGD}, we provide an explicit description of Kahan's automatic step-size control strategy \cite{kahn:2019b,kahn:2019a}.  Section \ref{sec:quad} focuses on the quadratic model and establishes the relationship between KGD step-sizes for  {Regime 1} and BB steps; particularly, for both BB  and KGD, we will present the corresponding long and short step-sizes for {Regime 1} and show that each enjoys the R-linear convergence in the strongly convex quadratic model.  Section \ref{sec:globalalgorithm} is devoted to the general unconstrained minimization in which we will propose an effective adaptive KGD step-size framework that dynamically treats the two regimes \eqref{eq:Regime01}; global convergence and local R-linear convergence analysis of the adaptive KGD method will be performed. Numerical tests on the CUTEst collection \cite{gort:2015}  {as well as the practical logistic regression problems \cite{mami:2020}  will be carried out in Section \ref{sec:numer}, and  the proposed methods are compared to various BB step-size approaches and   other recently proposed adaptive gradient methods \cite{mami:2020} to demonstrate the efficiency and robustness}. Final remarks are drawn in Section \ref{sec:conclu}.
    
%%%%%%%%%%%%%%%%%%%%%%%%%%%%%
\section{Kahan's automatic step-size control strategies}{\label{sec:KGD} }
    In this section we introduce Kahan's automatic step-size control strategies given in \cite{kahn:2019b,kahn:2019a}. 
    \par
    We first discuss the KGD step-size strategy for Regime 0.  Let $x$ be the current approximation, and suppose we already have a trial step-size $\alpha>0$ and have computed  an associated trial update      
    \begin{equation}\nonumber
         {\tilde{x}_{\rm new}(x,\alpha)}= x - \alpha G(x),\quad \mbox{with~~} G(x)=\nabla f(x).
    \end{equation}
When $x$ is not sufficiently close to a minimizer (i.e., $x$ is regarded as  situated in Regime 0 and our practical and precise rule for determining $x$ to be in  Regime 0 will be given in Algorithm \ref{alg:global}), the trial ${\tilde{x}_{\rm new}=\tilde{x}_{\rm new}(x,\alpha)}$  may not deliver improvement over $x$ in the sense that $f(\tilde{x}_{\rm new})>f(x)$. In this case, Kahan's idea is to shrink the trial step-size $\alpha$ by using the information of $\tilde{x}_{\rm new}$. In particular, {assuming $f$ is three times continuously differentiable at $x$ and} using Taylor's expansion, we have  
    \begin{equation}{\label{eq2.1}}
        f(\tilde{x}_{\rm new}) = f(x) - \alpha \cdot \frac{1}{2}\left[G(x)+G(\tilde{x}_{\rm new})\right]^{\T} G(x) +O(\alpha ^3 \|G(x)\|_2^{{3}}).
    \end{equation}
    Notice 
    \begin{align*}\nonumber
         &G(x)^{\T}G(\tilde x_{\rm new})= \frac14\left[\|G(\tilde{x}_{\rm new})+G(x)\|_2^2 -\|G(\tilde{x}_{\rm new})-G(x)\|_2^2\right], \\
         &\|G(\tilde{x}_{\rm new})-G(x)\|_2^2=O(\alpha^2 \|G(x)\|_2^2);
    \end{align*}
 combine them with \eqref{eq2.1}  to have
    \begin{align*}\nonumber
	f(\tilde{x}_{\rm new})&= f(x) - \alpha \cdot \frac{1}{2}\left[\|G(x)\|_2^2+G(\tilde{x}_{\rm new})^{\T} G(x)\right] +O(\alpha ^3 \|G(x)\|_2^{{3}})
	\\
	&=f(x)-\alpha \cdot \frac{1}{8}\left[\|G(x)+G(\tilde{x}_{\rm new})\|_2^2+4\|G(x)\|_2^2 \right] + O\left(\alpha ^3 \|G(x)\|_2^2\right).
    \end{align*}
    Hence we have a gauge function of   $f(x)$ at $x$ along the direction $-G(x)$ with a parameter $\Delta \tau\ge 0$:
    \begin{align}\nonumber
        f(x-\Delta \tau G(x)) =& f(x) -  \frac{\Delta \tau }{8}\left[\|G(x)+G(\tilde{x}_{\rm new})\|_2^2+4\|G(x)\|_2^2 \right] \\{\label{eq2.2}}
        &+ \Upsilon(x,\Delta \tau ) \Delta \tau^3 \|G(x)\|_2^2.
    \end{align}
    Assuming the function $\Upsilon(x,\Delta \tau)$ varies slowly in $\Delta \tau$ (i.e., assume $\Upsilon(x,\Delta \tau)\approx \Upsilon(x,\alpha)$ for the parameter $\Delta \tau$) and substituting  it with 
    $$\Upsilon(x,\alpha)=\frac{f(\tilde{x}_{\rm new})-f(x)+  \frac{\alpha}{8}\left(\|G(x)+G(\tilde{x}_{\rm new})\|_2^2+4\|G(x)\|_2^2\right)}{\|G(x)\|_2^2\cdot\alpha^3},$$   we have from \eqref{eq2.2} that 
    \begin{equation}{\label{eq2.3}}
            \begin{split}
			f(x-\Delta \tau G(x)) &\approx f(x)-  \frac{\Delta \tau}{8}[\|G(x)+G(\tilde{x}_{\rm new})\|_2^2+4\|G(x)\|_2^2]\\
			&+\left(\frac{\Delta \tau}{\alpha}\right)^3 \left (f(\tilde{x}_{\rm new})-f(x)+  \frac{\alpha}{8}\left(\|G(x)+G(\tilde{x}_{\rm new})\|_2^2+4\|G(x)\|_2^2\right)\right). 
            \end{split}
    \end{equation}
Based on \eqref{eq2.3},   Kahan's idea is to compute the new step-size $\alpha_{\new} $ for $x$ by   minimizing  the right hand side of \eqref{eq2.3}  which then yields %\marginpar{\tiny $\alpha_{\new}<\alpha$ always?}
    \begin{equation}\nonumber
        \alpha_{\new} = \frac{\alpha}{\sqrt{3+\frac{24[f(\tilde{x}_{\rm new})-f(x)]}{\alpha [\|G(x)+G(\tilde{x}_{\rm new})\|_2^2+4\|G(x)\|_2^2]}}}.
    \end{equation}
    To simplify our discussion, we define Kahan's automatic step-size in  {Regime 0} as follows:
    \begin{definition}{\label{KGD0}}
        For a given trial step $\alpha>0$ and a trial ${\tilde{x}_{\rm new} =\tilde{x}_{\rm new}(x,\alpha)}:= x - \alpha G(x)$, we define\footnote{{The first argument $x$ in  $\Kzero(x; x,\alpha)$ indicates that $\Kzero(x; x,\alpha)$  is a step-size for the iterate $x$, and $\Kzero(x; x,\alpha)$ is a function of the second  pair of arguments  ($x,\alpha$).}} 
    \begin{equation}{\label{eq:KGD0} }
            {\Kzero(x; x,\alpha)} := \frac{\alpha}{\sqrt{3+\frac{24[f(\tilde{x}_{\rm new})-f(x)]}{\alpha [\|G(x)+G(\tilde{x}_{\rm new})\|_2^2+4\|G(x)\|_2^2]}}}
        \end{equation}
        as the Kahan's automatic step-size in {Regime 0} for $x$ {with respect to the trial $\tilde{x}_{\rm new}$}.
    \end{definition}
    
More properties of {$\Kzero(x; x,\alpha)$} will be given in Lemmas \ref{lm4.1} and \ref{lm4.2}. In particular, it is worth mentioning here that whenever the trial ${\tilde{x}_{\rm new}(x,\alpha)}$ does not deliver a sufficient improvement in the sense of \eqref{lm1:cond} on the $f(x)$, then  {$\Kzero(x; x,\alpha)$} will shrink until \eqref{lm1:cond} is fulfilled.
  
  When $x$ is in {Regime 1},  local convergence to a minimizer can take place and $\|G(x)\|_2$ is small; thus suppose $\alpha>0$ is the step-size for $x$ and we have the new update (not a trial point) $\xnew=x-\alpha G(x)$. KDG step-size for Regime 1 is a strategy to set the step-size $\alpha_\new$ for $\xnew$. Precisely, in this case, a gauge function based on the quadratic approximation 
    \begin{equation}{\label{eq:estKGD1}}
	f(x-\Delta \tau G(x)) \approx f(x)-\Delta \tau \cdot \|G(x)\|_2^2 +\Delta \tau^2 \frac{1}{2} G(x)^{\T} H(x) G(x) 
    \end{equation}
    of  $f(x)$ at $x$ is sufficiently accurate. To avoid explicitly using the information of the Hessian $H(x)$, Kahan's idea is to reuse the information associated with $\xnew$. 
    Applying similarly the argument as \eqref{eq2.3},  the following  
    \begin{equation}\nonumber
        \alpha_{\new} = \frac{\alpha }{2+ \frac{2(f(\xnew)-f(x) )}{ \alpha \|G(x)\|_2^2} }
    \end{equation}
approximates well the minimizer that minimizes the right hand side of \eqref{eq:estKGD1} as \cite{kahn:2019b}
\begin{equation}\nonumber
\frac{\alpha }{2+ \frac{2(f(\xnew)-f(x) )}{ \alpha \|G(x)\|_2^2} }\approx \frac{\|G(x)\|_2^2}{G(x)^{\T}H(x)G(x)\pm \alpha \cdot O(\|G(x)\|_2^3)}.
\end{equation}
  Thus,  we similarly can define Kahan's automatic step-size in  {Regime 1}: 
    \begin{definition}{\label{KGD1}}
        For a step-size $\alpha>0$ and  $\xnew = x - \alpha G(x)$, we define\footnote{{The first argument $\xnew$ in   $\Kone(\xnew; x,\alpha)$ indicates that $\Kone(\xnew; x,\alpha)$ is  a step-size for the iterate $\xnew$, and $\Kone(\xnew; x,\alpha)$ is dependent on the second pair of arguments  ($x,\alpha$).}}
        \begin{equation}{\label{eq:KGD1} }
            {\Kone(\xnew; x,\alpha)} := \frac{\alpha }{2+ \frac{2(f(\xnew)-f(x) )}{ \alpha \|G(x)\|_2^2} }
        \end{equation}
        as the Kahan's automatic step-size in {Regime 1} for $\xnew$.
    \end{definition}
%    \begin{remark}
%        In Kahan's work \cite{kahn:2019b}, equation \eqref{eq:KGD0} aims to shrink the trial step to make $f(\xnew)<f(x)$ hold while the step-size in equation \eqref{eq:KGD1} is used directly in next iteration regardless of the decreasing condition. 
%    \end{remark}
%    \begin{remark}
%        It is known that the eigenvalue of the Hessian matrix $H(x)$ does not change for any convex quadratic equation. Based on the expression in \eqref{eq:Regime01}, the iteration process can be always seen in \textbf{Regime 1} and we will analyse the iteration of the form
%        \begin{equation}{\label{it:r1}}
%            \begin{cases}
%                 x_{k+1} = x_{k} - \alpha_{k} g_k \\
%                 \alpha_{k+1} = \Kone(x_k,\alpha_k)
%            \end{cases}
%        \end{equation}
%        in Section \ref{sec:quad}.
%    \end{remark}

%%%%%%%%%%%%%    
\section{Convergence of KGD step-sizes for the quadratic model}{\label{sec:quad} }
We first investigate the convergence of KGD step-size strategies for the strongly convex problem \eqref{eq:quadf}, where $H$ is  symmetric positive definite. In this case,  $G(x) = Hx+b$ and $x^* = - H^{-1} b$ is the solution. Moreover, as  the  Hessian matrix is constant,  based on our statements for Regime 0 and Regime 1 in \eqref{eq:Regime01},  the current $x$ can be assumed to be in Regime 1 and the iteration resulted from KDG step-size strategy  reads as 
        \begin{equation}{\label{it:r1}}
            \begin{cases}
                 x_{k+1} = x_{k} - \alpha_{k} G(x_k) \\
                 \alpha_{k+1} = {\Kone(x_{k+1};x_k,\alpha_k)}.
            \end{cases}
        \end{equation}
\par
To see more clearly the formulation of $\alpha_{k+1}$ in \eqref{it:r1},  by \eqref{eq:quadf}, we have 
\begin{equation}{\label{eq3.1} }
    f(x) - f(y) = \frac{1}{2} \left( G(x)+G(y) \right)^{\T} \left( x-y \right).
\end{equation}
Plug \eqref{eq3.1} into \eqref{eq:KGD1} to get 
\begin{equation}{\label{eq:BBl}}
    {\Kone(x_{k+1};x_{k},\alpha_k )} = \frac{s_k^{\T} s_k }{s_k^{\T} y_k }=\alpha_{k+1}^{\text{BB1}},
\end{equation}
where $s_k = x_{k+1} - x_k$ and $y = G(x_{k+1}) - G(x_{k})$. This shows the equivalence \cite{gjmx:2023} between Kahan's step-size in {Regime 1} \eqref{eq:KGD1} and the long Borwein-Barzilai step \cite{babo:1988} given in \eqref{BBl}.

\subsection{Long and short KGD step-sizes in Regime 1}\label{subsec:longshortKGD}
Due to the relation \eqref{eq:BBl}, we call the step-size $\Kone(x_{k},\alpha_k )$ in \eqref{eq:KGD1} as the long KGD step-size for $x_{k+1}$  in Regime 1. 
Inspired by this equivalence, we next provide a short  KGD step-size  in Regime 1 which reduces to  the short  BB step \eqref{BBs} for the quadratic model. Precisely, for a step-size $\alpha>0$ and  $\xnew = x - \alpha G(x)$, we define the short  KGD step-size  in Regime 1 for $\xnew$ to be
\begin{equation}{\label{eq:shortKGD} }
     {\Kone^{\rm s}(\xnew; x,\alpha)}:= \frac{2(\alpha \|G(x)\|_2^2 + (f(\xnew) - f(x)) ) }{ \|G(\xnew)-G(x)\|_2^2 }.
\end{equation}
To see the equivalence in quadratic model \eqref{eq:quadf}, we can plug \eqref{eq3.1} into \eqref{eq:shortKGD}  to get 
\begin{equation}{\label{eq:BBs}}
     {\Kone^{\rm s}(\xnew; x_k,\alpha_k)}= \frac{ s_k^{\T} y_k}{ y_k^{\T} y_k }=\alpha_{k+1}^{\text{BB2}}.
\end{equation}
 
%Now we have four different step-sizes available in \textbf{Regime 1} described in \eqref{eq:KGD1}, \eqref{eq:shortKGD}, \eqref{eq:BBl} and \eqref{eq:BBs}. We will call them \textbf{Long KGD step}, \textbf{Short KGD step}, \textbf{Long BB step} and \textbf{Short BB step} respectively in order. Without ambiguity, the term $\Kone(x,\alpha)$ denotes one of the four step-sizes above in the rest part of our article. 
%

\subsection{Convergence in the quadratic model}\label{subsec:convgquad}
According to the relationship \eqref{eq:BBl}, in exact arithmetic, we know that for the quadratic model, the iteration sequence generated by using the long BB step \eqref{BBl} is equivalent to that from using the long KGD step \eqref{eq:BBl}, provided that they start from the initial $x_0$ and the same step-size $\alpha_0$. Analogously, by \eqref{eq:BBs}, this is true for the short versions.  This means that all convergence behaviors of the long and short BB steps can be applied  to the long and short KGD steps. %In particular, a recent work in \cite{lisu:2021} has proved the following result for the long BB step. 
{
\begin{theorem}{\label{thm3.1} }
    Let $f(x)$ be a strongly convex quadratic function \eqref{eq:quadf}. Then, 
    \begin{itemize}
    \item[{\rm (i)}] {\rm (\cite{rayd:1993})} Both the long and short BB gradient methods  converge  to the minimizer $x^* = -H^{-1} b$ of $f(x)$. 
    \item[{\rm (ii)}]  {\rm (\cite{lisu:2021})} Moreover, the convergence of the BB long step-size \eqref{BBl} method is at least linear with a rate of $1- \frac{1}{\kappa}$, where $\kappa$ is the condition number of $H$.  
    \end{itemize}
\end{theorem}
}
We next consider  short step-size strategies for BB and KGD in quadratic models. Let $\{x_k\}$ be the sequence generated by the short KGD (equivalently, the short BB step \eqref{BBs}) step  \eqref{eq:BBs}. By calculations, we have  
\begin{equation}\nonumber
    \alpha_{k+1}={\Kone^{\rm s}(x_{k+1};x_k , \alpha_k)} = \frac{s_k^{\T} H s_k }{s_k^{\T} H^2 s_k }
\end{equation}
with $s_k = x_{k+1} - x_k$. Introducing $w_k = H^{\frac{1}{2}} x_k$, we have
\begin{equation}\nonumber
    w_{k+1} = w_k - \alpha_k (Hw_k + H^{\frac{1}{2}}b ).
\end{equation}
Now construct a new quadratic function $\varphi(w) = \frac{1}{2} w^{\T} H w + (H^{\frac{1}{2}} b )^{\T} w$ and denote its gradient by $\Phi(w)$; then we have
\begin{equation}\nonumber
    w_{k+1} = w_k - \alpha_k \Phi(w).
\end{equation}
It is interesting to note that
\begin{equation}\nonumber
    \alpha_{k+1}  =\frac{z_k^{\T} z_k }{z_k^{\T} H z_k}={ \Kone(w_{k+1};w_k , \alpha_k)}
\end{equation}
with $z_k = w_{k+1}-w_k$. This shows that {\it the sequence $\{x_k\}$ obtained from using the short KGD step \eqref{eq:BBs} is exactly corresponding to the sequence $\{w_k\}$ for the counterpart quadratic model $\varphi(w)$ with the long KGD (equivalently, the long BB) step \eqref{eq:BBl}.} By this relation and Theorem \ref{thm3.1}, we conclude  
\begin{theorem}{\label{thm3.2} }
    Let $f(x)$ be a strongly convex quadratic function \eqref{eq:quadf}. Then, in exact arithmetic, the sequences 
     $\{x_k\}$  generated from the short KGD step $\alpha_{k+1}=\Kone^{\rm s}(x_{k+1};x_k,\alpha_k)$ given  in \eqref{eq:shortKGD} and from the short BB step  $\alpha_{k+1}^{\rm BB2}$ given in \eqref{BBs} are the same; furthermore, the sequence $\{x_k\}$ converges to the minimizer $x^* = -H^{-1} b$ of $f(x)$ at least linearly with a rate of $1- \frac{1}{\kappa}$, where $\kappa$ is the condition number of $H$.
\end{theorem}

Due to Theorems \ref{thm3.1} and \ref{thm3.2}, to simplify our presentation, we introduce the operator  ${\mathsf K}_1$ to represent either the long (i.e., ${\mathsf K}_1=\Kone$) or the short  (i.e., ${\mathsf K}_1=\Kone^{\rm s}$) KGD step-size updating.  Correspondingly, we have the iteration 
        \begin{equation}{\label{eq:lsKGD}}
            \begin{cases}
                 x_{k+1} = x_{k} - \alpha_{k} G(x_k) \\
                 \alpha_{k+1} = {{\mathsf K}_1(x_{k+1};x_k,\alpha_k)}.
            \end{cases}
        \end{equation}

%%%%%%%%%%%%%%%%%%%%%%%%%%%%%%%%%%%%
\section{Adaptive KGD method for unconstrained minimization}{\label{sec:globalalgorithm} }
In this section, we will extend our discussion on the quadratic model to the general unconstrained optimization. To deal with a general $f(x)$, effective and precise rules should be specified in a framework to determine the regime where the current iterate $x_k$ is situated, and be able to adaptively choose the proper KGD size-size formula to  ensure the global convergence. 

{Let $x_0$ be the initial point. We also make the following assumption:
\begin{assumption}\label{assum:A1}
    The function $f:\bbR^n \to \bbR$ is  {three times} continuously differentiable, and 
     $\Omega = \left\{ x\in \bbR^n : f(x) \leq f(x_0) \right\}$ is compact. 
     \end{assumption}
With this assumption, it follows that there is  a positive constant $\Lambda_1$ such that
    \begin{equation}{\label{A1} }
        v^{\T} H(x) v \leq \Lambda_1 \|v\|_2^2
    \end{equation}
    holds for any $v \in \bbR^n$ and $x \in \Omega$, where $H(x)= \nabla^2 f(x)$ is the Hessian matrix at $x$; also, 
there exists a Lipschitz constant $L>0$ such that
    \begin{equation}{\label{A2} }
        \|H(x) - H(y)\|_2 \leq L\|x-y\|_2.
    \end{equation}
    }
%    holds for any $x,y \in \Omega$.
%
%
% and suppose $\Omega = \left\{ x\in \bbR^n : f(x) \leq f(x_0) \right\}$ to be compact. We also make the following assumptions:
%\begin{assumption}\label{assum:A1}
%    The function $f:\bbR^n \to \bbR$ is {three times} continuously differentiable, and there exists a positive constant $\Lambda_1$ such that
%    \begin{equation}{\label{A1} }
%        v^{\T} H(x) v \leq \Lambda_1 \|v\|_2^2
%    \end{equation}
%    holds for any $v \in \bbR^n$ and $x \in \Omega$, where $H(x)= \nabla^2 f(x)$ is the Hessian matrix at $x$.
%\end{assumption}
%
%\begin{assumption}\label{assum:A2}
%    There exists a Lipschitz constant $L>0$ such that
%    \begin{equation}{\label{A2} }
%        \|H(x) - H(y)\|_2 \leq L\|x-y\|_2
%    \end{equation}
%    holds for any $x,y \in \Omega$.
%\end{assumption}
%

    The following lemmas concern with ${\Kzero(x;x,\alpha)}$ and will be used later.
\begin{lemma}{\label{lm4.1} }
    Let $x \in \Omega$ with $G(x) \neq 0$ and $\alpha > 0$. Consider the following sequence
    \begin{equation}{\label{lm1:seq} }
        \begin{cases}
        \Delta \tau_0 = \alpha\\
            \Delta \tau_{k+1} = {\Kzero(x; x,\Delta \tau_{k})}.
                    \end{cases}
    \end{equation}
    Then there exists an integer $p \geq 0$ such that
    \begin{equation}{\label{lm1:cond} }
        f(x - \Delta \tau_{p} G(x) ) \le f(x) - \eta \Delta \tau_{p} \|G(x)\|_2^2,
    \end{equation}
    where $0< \eta < \frac{1}{3}$ is a given parameter.
\end{lemma}
\begin{proof}
For the given $\eta$, suppose   $\Delta \tau > 0$ satisfies 
    \begin{equation}\nonumber\label{eq:noimprove}
        f(\xnew) \geq f(x) - \eta \Delta \tau \|G(x)\|_2^2,
    \end{equation}
    where $\xnew = x - \Delta \tau G(x)$. Note  
        \begin{equation}\nonumber
        \begin{split}
           &3+ \frac{24[f(\xnew)-f(x)]}{\Delta \tau [\|G(x)+G(\xnew) \|_2^2 + 4\|G(x)\|_2^2 ]}\\
           \geq & \,  3- \frac{ 24\eta\|G(x)\|_2^2 }{\|G(x)+G(\xnew) \|_2^2 + 4\|G(x)\|_2^2 }\\
            \geq & \, 3(1-2 \eta).
        \end{split}
    \end{equation}
    Hence
    \begin{equation}\nonumber
        {\Kzero(x;x,\Delta \tau)} \leq \frac{1 }{ \sqrt{ 3(1-2\eta)} } \Delta \tau.
    \end{equation}
    This fact indicates that the step-size in \eqref{lm1:seq} shrinks to zero at a rate of $ \frac{1 }{ \sqrt{ 3(1-2\eta)} } < 1$ whenever the condition \eqref{lm1:cond} invalidates. To see   \eqref{lm1:cond} can be met after a finite number of iterations, we note, by the   mean-valued theorem (see \cite[Equ. (A56)]{nowr:2006}), that locally,
 \begin{align*}\nonumber
        f(x - t G(x) ) &= f(x)-t\|G(x)\|_2^2+\frac{t^2}{2} G(x)^{\T}H(x - \hat t\cdot  tG(x)) G(x) ~~({\rm for~ some~}\hat t\in [0,1])\\
        &\le  f(x)-t\|G(x)\|_2^2+\frac{t^2\Lambda_1}{2} \|G(x)\|_2^{2}\quad\quad \quad \quad \quad  ({\rm by ~} \eqref{A1})\\
        &\le  f(x) - \eta  t \|G(x)\|_2^2,
    \end{align*}
for all $0\le t\le \frac{2(1-\eta)}{\Lambda_1}$. This implies that after a finite number of iterations of \eqref{lm1:seq}, the condition \eqref{lm1:cond} will be fulfilled.
\end{proof}
\begin{lemma}{\label{lm4.2}}
    Let $\xi = \sup_{x \in \Omega} \|G(x)\|_2$. Then for any $x \in \Omega$ and $\Delta \tau > 0$, we have
    \begin{equation}\nonumber
        {\Kzero(x;x, \Delta \tau)} \geq \frac{2 }{ \sqrt{12\xi L+3\Lambda_1} }.
    \end{equation}
\end{lemma}
\begin{proof}
    With $\xnew = x - \Delta \tau G(x)$, by the mean-valued theorem, we have 
     \begin{align*}
     &f(\xnew)-f(x)- G(x)^{\T}(\xnew -x) =\frac12(\xnew -x)^{\T}H(x+t_1(\xnew -x))(\xnew -x),\\
       & (G(\xnew)-G(x))^{\T}(\xnew -x) =  (\xnew -x)^{\T}H(x+t_2(\xnew -x))(\xnew -x),
     \end{align*}
     for some $t_1,t_2\in [0,1]$. Thus 
     \begin{equation}\nonumber
        \left|f(\xnew)-f(x)-\frac{1}{2}(G(x)+G(\xnew))^{\T}(\xnew -x) \right| \leq \frac{L}{2} \Delta \tau^3 \|G(x)\|_2^3.
    \end{equation}
    Combining it with \eqref{A1} and using 
    \begin{align*}
    &\frac{\Delta \tau}{8} [\|G(x)+G(\xnew)\|_2^2+4\|G(x)\|_2^2]-\frac{1}{2}(G(x)+G(\xnew))^{\T}(\xnew -x)\\
  =  &\frac{\Delta \tau}{8}  \|G(\xnew)-G(x)\|_2^2,
    \end{align*}
    we have
             \begin{equation}\nonumber
                 \begin{split}
                    &\left| f(\xnew)-f(x) +\Delta \tau \cdot \frac{1}{8} [\|G(x)+G(\xnew)\|_2^2+4\|G(x)\|_2^2] \right|\\
                    \leq &\,  \left|f(\xnew)-f(x)-\frac{1}{2}(G(x)+G(\xnew))^{\T}(\xnew -x) \right| + \frac{1}{8} \Delta \tau \|G(\xnew)-G(x)\|_2^2\\
                    \leq &\, \frac{L}{2} \Delta \tau^3 \|G(x)\|_2^3 + \frac{\Lambda_1}{8 } \Delta \tau^3 \|G(x)\|_2^2\\
                    \leq &\, \left(\frac{\xi L}{2} + \frac{\Lambda_1}{8} \right) \Delta \tau^3 \|G(x)\|_2^2.
                 \end{split}
             \end{equation}
    Hence we have
    \begin{equation}\nonumber
        \begin{split}
        {\Kzero(x;x,\Delta \tau)}  & = \frac{\Delta \tau }{\sqrt{ 3+ \frac{24[f(\xnew)-f(x)]}{\Delta \tau [\|G(x)+G(\xnew) \|_2^2 + 4\|G(x)\|_2^2 ]} } } \\
                    & =  \, \frac{\Delta \tau \sqrt{\Delta \tau [\|G(x)+G(\xnew) \|_2^2 + 4\|G(x)\|_2^2] } }{\sqrt{ 24(f(\xnew)-f(x))+ 3 \Delta \tau [\|G(x)+G(\xnew) \|_2^2 + 4\|G(x)\|_2^2] }}\\
                    & \geq \frac{\sqrt{ 4\|G(x)\|_2^2 } }{\sqrt{ 24[(\frac{\xi L}{2} + \frac{\Lambda_1}{8} ) \|G(x)\|_2^2 ] } } \\
                    & =  \frac{2}{\sqrt{12\xi L+3\Lambda_1} }
        \end{split}
    \end{equation}
    as desired.
\end{proof}

%%%%%%%%%%
\subsection{Adaptive KGD methods and the global convergence} \label{subsec:KGDadp}
    We first point out that the iteration scheme \eqref{it:r1} alone might not work even in the strictly  convex case. Indeed, this is the case for BB steps: in the review paper by Fletcher \cite{flet:2005}, it is claimed that the BB method diverges  for certain initial points in the test problem referred to as {\em strongly convex 2} by Raydan \cite{rayd:1997}:    \begin{equation}\nonumber
        f(x) = \sum_{i=1}^n i(e^{x_i}-x_i)/10.
    \end{equation}
    For KGD, we will present a new example on which the iteration \eqref{it:r1} does not converge.  {Similar to \cite{budh:2019},} we consider the following  function
        \begin{equation}{\label{fun:counterexample} }
			f(x)=\begin{cases}
				(x-b)^2, \, x\leq -a,\\
				c_1 x^4 +c_2 x^2 + c_3 , \, -a<x<a, \\
				(x+b)^2, \, x\geq a,
			\end{cases}
	\end{equation}
    where $0<a<b$ and $c_1,c_2,c_3$ are chosen to make  $f(x)$ smooth enough.  Note that the gradient of $f(x)$ is 
        \begin{equation}\nonumber{\label{fun:countergrad} }
		g(x)= \begin{cases}
			2(x-b) , \, x\leq -a, \\
			4c_1 x^3 +2c_2 x , \, -a<x<a, \\
			2(x+b) , \, x \geq a,
		\end{cases}
	\end{equation}
    which is an odd and monotonically increasing continuous function. Choose $x_0 = -b$ and $x_1 = -a$ to be the initial points, and one can verify $\alpha_1 = \frac{1}{2}$ and $x_2 = b$ for any four step-size strategies. With a proper choice of $0<a<b$, we can fulfill the condition ${\Kone(x_2;b,\frac{1}{2})} = \frac{b-a}{4b}$ by solving an algebraic equation of $\frac{a}{b}$ with a unique root between $0$ and $1$. It then can be seen  that  $x_3 = b - \frac{b-a}{4b} \cdot 4b = a$. Using the symmetry of $f(x)$, we get $x_4 = -b$ and $x_5 = -a$, and consequently,  we have  the cycle  between four points: $-b \to -a \to b \to a \to -b \to -a \to \cdots$ (see Figure \ref{fig:ce}).
    \begin{figure}[h!!!]
    \centering
    \subfloat{\includegraphics[width = 2.5in]{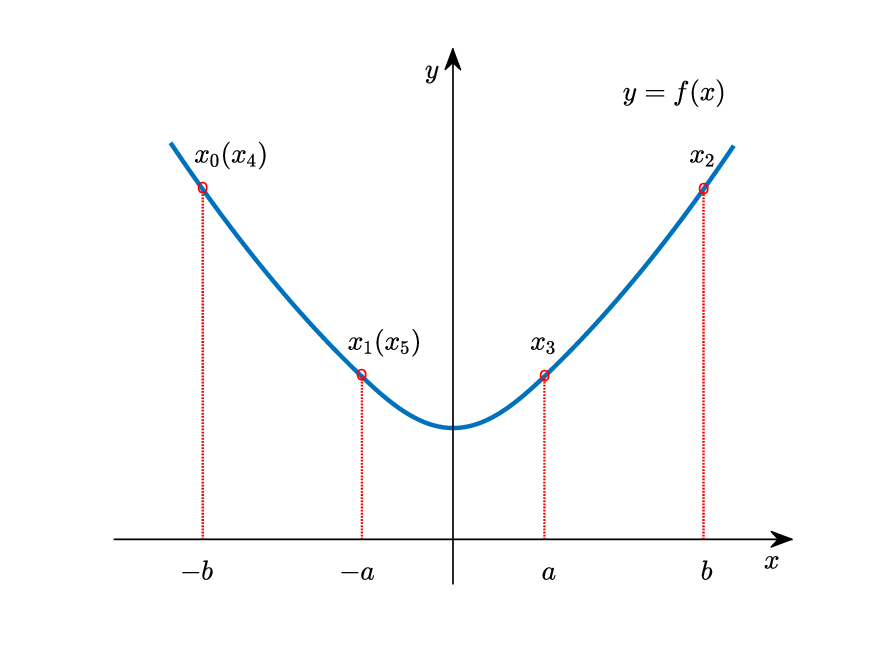} \label{fx}}
    \hfill
    \subfloat{\includegraphics[width = 2.5in]{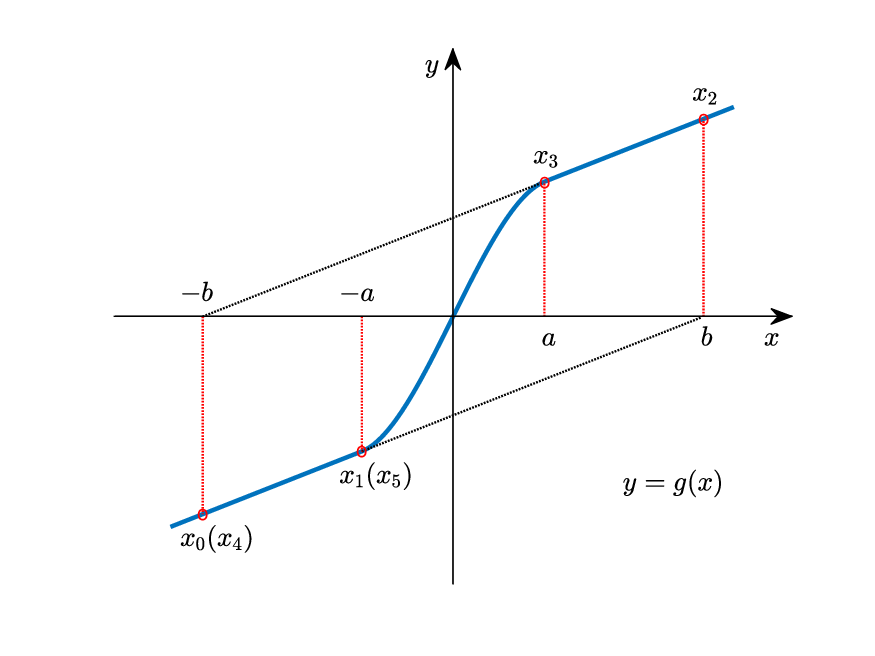} \label{gx}}
    \hfill
    \caption{\em  Illustration of the cycle  between four iterations from the long KGD step-size for the function \eqref{fun:counterexample}.} %The function $f(x)$ \eqref{fun:counterexample} (left) and its gradient \eqref{fun:countergrad} $g(x)$ (right).}
    \label{fig:ce}
    \end{figure}
    %This counterexample shows the necessity of a general framework. A useful way available is to apply a line-search technique. 
    
In order to stabilize the KGD step-sizes for the global convergence, at the current iterate $x_k$, we propose a simple rule (a nonmonotone Armijo-like line-search technique \cite{grip:1986}) to determine whether the KGD step-size for Regime 1 should be applied. In particular, for $x_k$, we continuously  use ${\Kzero(x_k;x_k,\alpha)}$ to update $\alpha\leftarrow{\Kzero(x_k;x_k,\alpha)}$ whenever the current $\alpha$ satisfies 
    \begin{equation}\label{eq:ruleKGD0}
        f(x_k-\alpha g_k) {\geq} \max_{0 \leq j \leq \min(k,M) } f(x_{k-j}) - \eta \alpha \|g_k\|_2^2,
    \end{equation}
     where $M$ is a fixed positive integer and $0 < \eta < \frac{1}{3}$. This adaptation works due to Lemma \ref{lm4.1}. {Other adaptive gradient-based methods for convex problems can be found in \cite{mami:2020,lats:2023}.} As long as \eqref{eq:ruleKGD0} invalidates for $(x_k,\alpha_k)$, {we then update the iterate $x_{k+1} = x_k - \alpha_k g_k$ and choose either the long KGD step-size $\alpha_{k+1} = \Kone(x_{k+1};x_k,\alpha_k)$ or the short KGD step-size $\alpha_{k+1} = {\Kone^{\rm s}(x_{k+1};x_k,\alpha_k)}$ as the initial step-size for $x_{k+1}$.}  The overall framework of our proposed adaptive KGD method is summarized in Algorithm \ref{alg:global}.
     \begin{remark}
      We remark that the framework of Algorithm \ref{alg:global} facilitates the use of other step-size strategies in Step 5 and Step 7. For example, the long/short BB step-size can be used in Step  7 in Algorithm \ref{alg:global}. In order to distinguish versions with different step-size strategies, we introduce the notation ${\tt KGDadp}(\Kzero,\Kone)$
      to indicate the one where   $\Kzero$ and $\Kone$ are used in Step 5 and Step 7, respectively. Therefore, other specific implementations 
      $$
      {\tt KGDadp}(\Kzero,\Kone^{\rm s}),\quad  {\tt KGDadp}(\Kzero,{\rm BB1}), \quad {\tt KGDadp}(\Kzero,{\rm BB2})
      $$
      can be interpreted analogously, for which numerical experiments will be reported in Section \ref{subsec:compfoursteps}. 
     \end{remark}
\begin{algorithm}[h!!!]
    \caption{Adaptive KGD method ({\tt KGDadp})}
    \label{alg:global}
    \begin{algorithmic}[1]
    \STATE Set the first point $x_0$ and the first step $\alpha_0 >0$;
    \STATE Choose parameters $\epsilon > 0$, $0 < \eta < \frac{1}{3}$, an integer $M$ and $k = 0$;
    \WHILE{$\|g_k\|_2 > \epsilon$ }
        \WHILE{ $f(x_k-\alpha_k g_k) {>} \max_{0 \leq j \leq \min(k,M) } f(x_{k-j}) - \eta \alpha_k \|g_k\|_2^2$ }
        \STATE $\alpha_k \leftarrow {\Kzero(x_k;x_k,\alpha_k)}$;
        \ENDWHILE
        \STATE {Set $x_{k+1} = x_k - \alpha_k g_k$ and $\alpha_{k+1} = {\mathsf K}_1(x_{k+1};x_k,\alpha_k)$;}
        \STATE $k = k + 1$;
    \ENDWHILE
    \end{algorithmic}
\end{algorithm}

For the global convergence of Algorithm \ref{alg:global}, based on Lemmas \ref{lm4.1} and \ref{lm4.1}, we can apply the similar  argument in \cite[Section 3]{grip:1986}  to conclude
    \begin{theorem}\label{thm:global}%\marginpar{\tiny check the statement.}
  The sequence $\{x_k\}$ from Algorithm \ref{alg:global} stays in $\Omega$, and {with the tolerance  $\epsilon=0$ in Step $3$,}  every limit point $x^*$ is a stationary point (i.e., $G(x^*) = 0$). Moreover, if the number of the stationary points of $f(x)$ on $\Omega$ is finite, then  $\{x_k\}$ converges.
    \end{theorem}
\begin{proof}        {At the $k$th iteration, according to Lemma \ref{lm4.1}, we know that the step-size $\alpha_k$ in Step 5 shrinks so that  the condition  $$f(x_k-\alpha_k g_k)  {\le} \max_{0 \leq j \leq \min(k,M) } f(x_{k-j}) - \eta \alpha_k \|g_k\|_2^2$$
    can be fulfilled. With this computed $\alpha_k$, the next iteration is $x_{k+1}=x_k - \alpha_k g_k$ (Step 7). Note the $\alpha_{k+1} = {\mathsf K}_1(x_{k+1};x_k,\alpha_k)$ in Step 7 only provides an initial step-size for $x_{k+1}$.  For   outer iterations, let $f(x_{\ell(k)})=\max_{0\le j\le  \min(k,M) } f(x_{k-j})$, and we can follow the same proof  of \cite[Theorem in Section 3]{grip:1986} to conclude the result.}
    \end{proof}

\subsection{Local R-linear convergence}
By Theorem \ref{thm:global}, we have known that every limit point of  {\tt KGDadp}  is a stationary point $x^*$.  Because a saddle point is numerically unstable, and the condition \eqref{eq:ruleKGD0} also ensures that the iterate $x_{k+1}$ delivers sufficiently improvement in the previous $M$ steps, a limit point $x^*$ of the sequence $\{x_k\}$ is most likely a local minimizer. In this subsection, we shall perform a local convergence analysis of  {\tt KGDadp} around a  minimizer $x^*$ with additionally the following assumption:

    \begin{assumption}\label{assum:A3}
    There exist a radius $\rho>0$ and a positive constant $\Lambda_2 > 0$ such that
    \begin{equation}{\label{A3} }
        v^{\T} H(x) v \geq \Lambda_2 \|v\|_2^2
    \end{equation}
    for any $v\in \bbR^n$ and $x \in B_{\rho} (x^*)$, where $x^*$ is a minimizer of $f(x)$.
\end{assumption}
 
Relying upon the equivalence (see \eqref{eq:BBl} and \eqref{eq:BBs}) between the long/short KGD step-size and  the long/short BB step-size for the quadratic model, we can use the same asymptotic technique developed in \cite{dahs:2006} and \cite{budh:2019} to show the R-linear convergence of {\tt KGDadp}.  In particular, denote $ H = H(x^*)$  and  construct a strongly quadratic approximation model 
    \begin{equation}\nonumber
        \fhat(x) = \frac{1}{2} (x-x^*)^{\T} H (x-x^*) + f(x^*),
    \end{equation}
 whose  gradient is
    \begin{equation}\nonumber
        \Ghat(x) = H (x - x^*).
    \end{equation}
Correspondingly, we define a sequence $\{\xhat_{k,j}\}_{j=0}^{\infty}$ by
    \begin{equation}\label{eq:subKGDseq}
        \begin{cases}
            \xhat_{k,0}=x_k,\\
            \xhat_{k,j+1}=\xhat_{k,j} - \ahat_{k,j} \Ghat(\xhat_{k,j}), \, j\geq 0,
        \end{cases} ~ \mbox{where}~\ahat_{k,j} = 
        \begin{cases}
        \alpha_k, \text{ if $j=0$} \\
        {\Konehat(\xhat_{k,j};\xhat_{k,j-1},\ahat_{k,j-1})}, \text{ if $j>0$},
        \end{cases}
    \end{equation}
%    where
%    \begin{equation}\nonumber
%        \ahat_{k,j} = 
%        \begin{cases}
%        \alpha_k, \text{ if $j=0$} \\
%        \Konehat(\xhat_{k,j-1},\ahat_{k,j-1}), \text{ if $j>0$},
%        \end{cases}
%    \end{equation}
    and $\Konehat$ is either the long or the short  KGD step-size operator applied on $\fhat$. For example, if the long KGD step is used, then we have
    \begin{equation}\nonumber
        {\Konehat(\xhat_{k,j};\xhat_{k,j-1},\ahat_{k,j-1}) =\Kone(\xhat_{k,j};\xhat_{k,j-1} ,\ahat_{k,j-1}) }=\frac{\ahat_{k,j-1} }{2+ \frac{2(\fhat(\xhat_{k,j})-\fhat(\xhat_{k,j-1} ) ) }{\ahat_{k,j-1} \|\Ghat(\xhat_{k,j-1} ) \|_2^2 }  },
    \end{equation}
   where, for convenience, we have defined $\shat_{k,j}= \xhat_{k,j}-\xhat_{k,j-1}$ for $j > 0$. Since $\fhat(x)$ is strongly convex and $x^*$ is its minimizer, we have
    \begin{proposition}
        There exists a positive integer $N$ dependent only on $\Lambda_1$ and $\Lambda_2$, such that for any $x_{k} \in \bbR^n$ it holds
        \begin{equation}\nonumber
            \|\xhat_{k,j}-x^*\|_2\leq \frac{1}{2}\|x_{k}- x^*\|_2, \, \forall j\ge N.
        \end{equation}
    \end{proposition}
   \begin{proof}
   Apply Theorems \ref{thm3.1} and \ref{thm3.2} to the sequence $\{\xhat_{k,j}\}_{j=0}^\infty$ to get $\lim_{j\rightarrow \infty}\xhat_{k,j}= x^*$ and the conclusion follows.
   \end{proof}

   The following two propositions can also be verified readily.
    \begin{proposition}
        Under  Assumptions \ref{A1} and \ref{assum:A3}, for any $x,y\in B_{\rho} (x^*)$, we have
        \begin{equation} {\label{prop:1} }
            \Lambda_2\|x-x^*\|_2\leq \|G(x)\|_2 \leq \Lambda_1 \|x-x^*\|_2,
        \end{equation}
        \begin{equation}{\label{prop:1'} }
            \frac{\Lambda_2}{2} \|y-x\|_2^2 \leq  f(y)-f(x) - G(x)^{\T}(y-x)   \leq \frac{\Lambda_1}{2} \|y-x \|_2^2,\, 
        \end{equation}
        \begin{equation}{\label{prop:2} }
            \|G(x)-H(x-x^*) \|_2\leq \frac{L}{2}\|x-x^*\|_2^2
        \end{equation}
        and
        \begin{equation} {\label{prop:3} }
            \left|f(x) - f(x^*)-\frac{1}{2} (x-x^*)^{\T} H (x-x^*)\right| \leq \frac{L}{6}\|x-x^*\|_2^3.
        \end{equation}
    \end{proposition}
    \begin{proposition}
        Under {{\bf Assumptions} \ref{assum:A1} and \ref{assum:A3}},  for any $x \in B_{\rho}(x^*)$ and $\Delta \tau>0$ satisfying  $\xnew =x-\Delta \tau G(x) \in B_{\rho} (x^*)$, we have  
        \begin{equation}\label{eq:bndK1}
            \frac{1}{\Lambda_1} \leq {\Kone(\xnew ;x,\Delta \tau)} \leq \frac{1}{\Lambda_2}.
        \end{equation}
    \end{proposition}

The proof of the local R-linear convergence of $\{x_k\}$ follows exactly the argument in \cite[Theorem 2.3]{dahs:2006} for the BB step, in which   a  key lemma (see the counterpart lemma in \cite[Lemma 2.2]{dahs:2006})  is as follows. 
\begin{lemma}{\label{lm1} }
       Under {{\bf Assumptions} \ref{assum:A1} and \ref{assum:A3}},  for any fixed positive integer $N$, there exist positive constants $\delta$ and {$c_1$} with the following property: for any $x_k \in B_{\delta} (x^*)$ , ${m}\in [0,N]$ with
        \begin{equation}{\label{condition} }
            \|\xhat_{k,j} - x^*\|_2 \geq \frac{1}{2} \|\xhat_{k,0} - x^* \|_2 \text{ for all $j \in [0,\max \{ 0,{m}-1 \}] $},
        \end{equation}
       it holds that 
        \begin{equation}\label{C1}
            x_{k+j} \in B_{\rho}(x^*)\quad \mbox{and}~~ \|x_{k+j}-\xhat_{k,j}\|_2\leq {c_1} \|x_k - x^*\|_2^2, \quad \forall j\in [0,{m}],
        \end{equation}
        where the sequence $\{x_{k+j}\}$ is from \eqref{eq:lsKGD} while  $\{\xhat_{k,j}\}$ is from \eqref{eq:subKGDseq}.
    \end{lemma} %\marginpar{\tiny please revise the proof according to Comment 4 of Referee \#1.}
    \begin{proof} The proof closely follows  that of \cite[Lemma 2.2]{dahs:2006}  with only proper modifications. 
    We prove  all the following inequalities together with \eqref{C1}:
    \begin{align}
            \left|f(x_{k+j})-\fhat(\xhat_{k,j})\right| &\leq {c_2} \|x_k-x^*\|_2^3, {\label{C2}} \\
            \|g_{k+j} - \ghat_{k,j}  \|_2& \leq {c_3} \|x_k - x^*\|_2^2, {\label{C3} } \\
            \|s_{k+j}\|_2 &\leq {c_4}\|x_k - x^*\|_2, {\label{C4} } \\
            |\alpha_{k+j}- \ahat_{k,j}| &\leq {c_5}\|x_k - x^*\|_2, {\label{C5} }
        \end{align}
        for $x_k \in B_{\delta} (x^*)$, where {$c_i$ are dependent on $N$, $\Lambda_1$, $\Lambda_2$, $L$ for $i=1,\dots,5$.}

{To prove  \eqref{C1}-\eqref{C5}, we apply the induction argument and introduce $c_{i,m}~(i=1,\dots,5)$ to represent the constants corresponding to  \eqref{C1}-\eqref{C5}, respectively, with the induction index $m$.
 For $m=0$, } taking $\delta = \rho$ leads to \eqref{C1}, and by \eqref{prop:3}, 
        %\begin{equation}
           $ |f(x_k) - \fhat(\xhat_{k,0}) | \leq \frac{L}{6} \|x_k - x^*\|_2^3,$
        %\end{equation}
        it yields  \eqref{C2}. Furthermore  based on \eqref{prop:2},  we have 
        %\begin{equation}
           $ \|g_k-\ghat_{k,0}\|_2\leq \frac{L}{2}\|x_k - x^*\|_2^2,$
        %\end{equation}
        and  thus \eqref{C3} holds. \eqref{C4} is true due to  $\delta = \rho$, $x_k \in B_{\delta}(x^*)$ and   
        \begin{equation}\nonumber
            \|s_k\|_2 = \|\alpha_k g_k\|_2 \leq \frac{\Lambda_2}{\Lambda_1} \|x_k - x^*\|_2. \quad \mbox{(by ~\eqref{prop:1})}
        \end{equation}
Finally,  \eqref{C5} holds with $\alpha_k=\ahat_{k,0}$.  {Thus, for $m=0$,  \eqref{C1}-\eqref{C5} hold with $c_{1,0}=0$, $c_{2,0}=\frac{L}{6}$, $c_{3,0}=\frac{L}{2}$, $c_{4,0}=\frac{\Lambda_2}{\Lambda_1}$ and $c_{5,0}=0$.}

Now, we assume  the induction hypothesis:  there exist $m\in [1,N)$  (we have proved the case $m=0$) and $\delta >0$ so   that if \eqref{condition} is true  $\forall j\in [0,m-1]$, then \eqref{C1}-\eqref{C5} hold    $\forall j\in [0,m]$. {That is, we have
\begin{subequations}\label{eq:C15}
    \begin{align}\label{eq:C1}
            \|x_{k+j}-\xhat_{k,j}\|_2 &\leq c_{1,m} \|x_k - x^*\|_2^2, \\\label{eq:C2}
            \left|f(x_{k+j})-\fhat(\xhat_{k,j})\right| &\leq c_{2,m} \|x_k-x^*\|_2^3,  \\\label{eq:C3}
            \|g_{k+j} - \ghat_{k,j}  \|_2& \leq c_{3,m} \|x_k - x^*\|_2^2,  \\\label{eq:C4}
            \|s_{k+j}\|_2 &\leq c_{4,m}\|x_k - x^*\|_2,  \\\label{eq:C5}
            |\alpha_{k+j}- \ahat_{k,j}| &\leq c_{5,m}\|x_k - x^*\|_2, 
        \end{align}
\end{subequations} 
for all  $x_k \in B_{\delta_{m}} (x^*)$, $j \in [0,m]$. The idea is to show that by choosing a smaller  $\delta >0$,  $m$ can be increased to $m+1$.}

Suppose  \eqref{condition} holds $\forall j \in [0,m]$. By the induction hypothesis and \eqref{eq:C4}, we have
        \begin{equation} {\label{eq:2} }
            \|x_{k+m+1}-x^*\|_2\leq \|x_k-x^*\|_2+ \sum_{i=0}^m\|s_{k+i}\|_2  \leq { (m+1)c_{4,m}}\|x_k - x^*\|_2,
        \end{equation}
        which implies that  by reducing {$\delta_{m+1} < \frac{\delta_m}{(m+1)c_{4,m}}$}, we can keep  $x_{k+m+1} \in B_{\rho}(x^*)$. 
        By the same argument  of the proof for \cite[Lemma 2.2]{dahs:2006}, \eqref{C1}, \eqref{C3} and \eqref{C4} hold for $m+1$. Next we will prove that \eqref{C2} and \eqref{C5} also hold for $m+1$.
        
        For \eqref{C2}, using the triangle inequality
        \begin{equation}\nonumber{\label{eq:3} }
                |f(x_{k+m+1})-\fhat(\xhat_{k,m+1} )|\leq |f(x_{k+m+1})-\fhat(x_{k+m+1})|+|\fhat(x_{k+m+1})-\fhat(\xhat_{k,m+1} )|,
        \end{equation}
        and by \eqref{prop:3} and \eqref{eq:2}, we have
        \begin{equation}\nonumber
            \begin{split}
            &\left|f(x_{k+m+1})-\fhat(x_{k+m+1})\right|\\
            =& \,\left|f(x_{k+m+1})-f(x^*)-\frac{1}{2}(x_{k+m+1}-x^*)^{\T}H (x_{k+m+1}-x^*)\right|\\
            \leq& \, \frac{L}{6} \|x_{k+m+1}-x^*\|_2^3\\
             \leq& \, {\frac{L(m+1)^3c_{4,m}^3}{6} }\|x_k-x^*\|_2^3;
            \end{split}
        \end{equation}
        moreover,  from \eqref{eq3.1},  \eqref{prop:1}, \eqref{eq:C1} and \eqref{eq:2}, we get
        \begin{equation}\nonumber
            \begin{split}
                &\left|\fhat(x_{k+m+1})-\fhat(\xhat_{k,m+1})\right|\\
                 \leq & \,\frac{1}{2}\left\|\Ghat(x_{k+m+1})+\Ghat(\xhat_{k,m+1})\right\|_2\cdot \|x_{k+m+1}-\xhat_{k,m+1}\|_2\\
                 \leq & \, \frac{1}{2} \Lambda_1 \left(\|x_{k+m+1}-x^* \|_2 + \|\xhat_{k,m+1} - x^* \|_2\right) \cdot \|x_{k+m+1}-\xhat_{k,m+1}\|_2\\
                 \leq & {\frac{\Lambda_1 \left( 2(m+1)c_{4,m}+c_{1,m+1}\delta_{m+1} \right)}{2}} \|x_k - x^*\|_2^3,
            \end{split}
        \end{equation}
        { where we have used
\begin{align*}
\left\|\hat{x}_{k, m+1}-x^*\right\|_2 &\leq\left\|\hat{x}_{k, m+1}-x_{k+m+1}\right\|_2+\left\|x_{k+m+1}-x^*\right\|_2 \\
& \leqslant c_{1, m+1}\left\|x_k-x^*\right\|_2^2+(m+1) c_{4, m}\left\|x_k - x^{*}\right\|_2 \\
& \leq\left(c_{1, m+1} \delta_{m+1}+(m+1) c_{4, m}\right)\left\|x_k-x^*\right\|_2.
\end{align*}
         Hence taking $c_{2, m+1}=\frac{L}{6}(m+1)^3 c_{4, m}^3+\frac{\Lambda_1}{2} \left(2(m+1) c_{4, m}+c_{1, m+1} \delta_{m+1}\right) c_{1, m+1}$
        establishes  \eqref{eq:C2} for $m+1$.}

        For \eqref{C5}, it suffices to prove it for the long KGD step-size,  because for the short KGD step-size, we can use the  relation between the long and short KGD step-sizes for the quadratic model established for Theorem \ref{thm3.2}. Note that
        \begin{equation}\nonumber
            \begin{split}
                \|s_{k+m}-\shat_{k,m}\|_2= &\, \|\alpha_{k+m}g_{k+m} - \ahat_{k,m}\ghat_{k,m} \|_2 \\
                \leq &\,\alpha_{k+m}\|g_{k+m}- \ghat_{k,m} \|_2+|\alpha_{k+m}-\ahat_{k,m}|\cdot\|\ghat_{k,m}\|_2\\
                \leq & {\frac{c_{4, m}}{\Lambda_2}}\left\|x_k-x^*\right\|_2^2+{c_{5, m} \Lambda_1}\left\|x_k-x^*\right\|_2^2 \\
                = & {\left(\frac{c_{4, m}}{\Lambda_2}+c_{5, m} \Lambda_1\right)}\|x_k-x^*\|_2^2.
            \end{split}
        \end{equation}
        Hence
        \begin{align}\nonumber
            \begin{split}
                \left|s_{k+m}^{\T} s_{k+m} - \shat_{k,m}^{\T} \shat_{k,m} \right|
                =& \, \left|2s_{k+m}^{\T} (s_{k+m} -\shat_{k,m})-\|s_{k+m}-\shat_{k,m}\|_2^2 \right|\\
                \leq& 2 {c_{4, m}\left(\frac{c_{4, m}}{\Lambda_2}+c_{5, m} \Lambda_1\right)}\left\|x_k-x^*\right\|_2^3+{\left(\frac{c_{4, m}}{\Lambda_2}+c_{5, m} \Lambda_1\right)^2}\left\|x_k-x^*\right\|_2^4\\
                \leq& {\underbrace{\left[2 c_{4, m}\left(\frac{c_{4, m}}{\Lambda_2}+c_{5, m} \Lambda_1\right)+\delta_m\left(\frac{c_{4, m}}{\Lambda_2}+c_{5, m} \Lambda_1\right)^2\right]}_{=:h}\left\|x_k-x^*\right\|_2^3}\\
                   \, =&h \|x_k - x^*\|_2^3.{\label{eq:4} }
            \end{split}
        \end{align}
        By \eqref{eq:bndK1}, $\ahat_{k,m}\in [\frac{1}{\Lambda_1}, \frac{1}{\Lambda_2}]$,  and we know
        %\begin{equation}\nonumber
           $ \|\shat_{k,m}\|_2 = \|\ahat_{k,m} \ghat_{k,m} \|_2 \geq \frac{\Lambda_2}{\Lambda_1} \|\xhat_{k,m}-x^*\|_2.$
        %\end{equation}
        Furthermore, \eqref{condition} implies  
        \begin{equation}{\label{eq:5} }
            \|\shat_{k,m}\|_2 \geq \frac{\Lambda_2}{2 \Lambda_1} \|x_k-x^*\|_2.
        \end{equation}
        Thus we get 
        \begin{equation}{\label{eq:7}}
            \begin{split}
            a = \left| 1-\frac{s_{k+m}^{\T}s_{k+m} }{\shat_{k,m}^{\T}\shat_{k,m}  } \right| = \frac{|s_{k+m}^{\T} s_{k+m} - \shat_{k,m}^{\T} \shat_{k,m} | }{ |\shat_{k,m}^{\T} \shat_{k,m}| } \leq \frac{2\Lambda_1 h}{\Lambda_2}\|x_k-x^*\|_2.
            \end{split}
        \end{equation}
        Notice 
        \begin{align}\nonumber
            %\begin{split}
                & |\left[ f(x_{k+m+1})-f(x_{k+m}) - g_{k+m}^{\T} (x_{k+m+1}-x_{k+m} )  \right] \\\nonumber
                &- \left[ \fhat (\xhat_{k,m+1})-\fhat(\xhat_{k,m}) - \ghat_{k,m}^{\T} (\xhat_{k,m+1}-\xhat_{k,m} )  \right]|\\\nonumber
                \leq & \, \left| f(x_{k+m+1})-\fhat (\xhat_{k,m+1})\right| +\left|f(x_{k+m})-\fhat(\xhat_{k,m})\right| \\\nonumber
                &+ \left|(g_{k+m}-\ghat_{k,m})^{\T} (x_{k+m+1}-x_{k+m} )\right| \\\nonumber
                &+ \left|\ghat_{k,m}^{\T}(x_{k+m+1}-\xhat_{k,m+1})\right|+\left|\ghat_{k,m}^{\T}(x_{k+m}-\xhat_{k,m} ) \right|\\{\label{eq:5'}}
                \leq &\, {\left(c_{2, m+1}+c_{2, m}+c_{3, m} c_{4, m}+\Lambda_1 c_{1, m+1}+\Lambda_1 c_{1, m}\right)}\|x_k-x^*\|_2^3,
            %\end{split}
        \end{align}
where the last inequality is due to \eqref{eq:C2}, \eqref{prop:1} and \eqref{eq:C1}. Also by \eqref{prop:1} and \eqref{prop:1'}, we have
        \begin{align}\nonumber
            %\begin{split}
            &\left|f(x_{k+m+1})-f(x_{k+m}) - g_{k+m}^{\T}(x_{k+m+1}-x_{k+m})\right| \\\nonumber
            \geq &\, \frac{\Lambda_2}{2}  \|x_{k+m+1}-x_{k+m}\|_2^2=\, \frac{\Lambda_2}{2} \alpha_{k+m}^2 \|g_{k+m}\|_2^2\\\nonumber
            \geq&\,  \frac{\Lambda_2^3}{2 \Lambda_1^2} \|x_{k+m}-x^*\|_2^2.
            %\end{split}
        \end{align}
        To bound $\|x_{k+m}-x^*\|_2^2$ from below,  according to our assumption  $\|\xhat_{k,m}- x^* \|_2\ge \frac12\|x_k-x^*\|_2$ on \eqref{condition}  and   \eqref{C1}, we can deduce (by choosing a sufficiently small  $\delta$)
        \begin{equation}\nonumber
            \|x_{k+m} - x^*\|_2 \geq \left|\|\xhat_{k,m}- x^* \|_2-\|x_{k+m}-\xhat_{k,m}\|_2 \right| \geq \frac14 \|x_k-x^*\|_2,
        \end{equation}
      and therefore,
        \begin{equation}{\label{eq:6} }
            \left|f(x_{k+m+1})-f(x_{k+m}) - g_{k+m}^{\T}(x_{k+m+1}-x_{k+m})\right| \geq {\frac{\Lambda_2^3}{32\Lambda_1^2}} \|x_k - x^*\|_2^2.
        \end{equation}
        Similar to  \eqref{eq:7} and using \eqref{eq:5'} and \eqref{eq:6}, one can get  
        \begin{align}\nonumber
           %  \begin{split}
             b=&\left| 1- \frac{ \fhat (\xhat_{k,m+1})-\fhat(\xhat_{k,m}) - \ghat_{k+m}^{\T} (\xhat_{k,m+1}-\xhat_{k,m} ) }{ f(x_{k+m+1})-f(x_{k+m}) - g_{k+m}^{\T}(x_{k+m+1}-x_{k+m}) }  \right| \\ \label{eq:8}
            \leq & \, {\frac{32\Lambda_1^2}{\Lambda_2^3}\left(c_{2, m+1}+c_{2, m}+c_{3, m} c_{4, m}+\Lambda_1 c_{1, m+1}+\Lambda_1 c_{1, m}\right)}\|x_k - x^*\|_2. 
                %\end{split}
        \end{align}
       By noticing $$\ahat_{k,m}^2\|\Ghat(\xhat_{k,m} )\|_2^2=\shat_{k,m}^{\T}\shat_{k,m},~~\alpha_{k+m}^2\|G(x_{k+m} )\|_2^2=s_{k+m}^{\T}s_{k+m},$$
       and
        $$\ahat_{k,m}\|\Ghat(\xhat_{k,m} )\|_2^2=-\ghat_{k,m}^{\T} (\xhat_{k,m+1}-\xhat_{k,m} ),~\alpha_{k+m}\|G(x_{k+m} )\|_2^2=-g_{k+m}^{\T} (x_{k+m+1}-x_{k+m} ),$$ the inequalities \eqref{eq:7} and \eqref{eq:8} then can be used to prove \eqref{C5} as  
        \begin{align}\nonumber
            %\begin{split}
                &|\ahat_{k,m+1}- \alpha_{k+m+1} | \\\nonumber
                =&\ahat_{k,m+1}\left|  1- \frac{1}{\ahat_{k,m+1}}\cdot \frac{\alpha_{k+m} }{2+ \frac{2(f(x_{k+m+1})-f(x_{k+m} ) ) }{\alpha_{k+m} \|G(x_{k+m} ) \|_2^2 } }\right|\\\nonumber
                %=&\, \left|  \frac{\ahat_{k,m} }{2+ \frac{2(\fhat(\xhat_{k,m+1})-\fhat(\xhat_{k,m} ) ) }{\ahat_{k,m} \|\Ghat(\xhat_{k,m} ) \|_2^2 } } - \frac{\alpha_{k+m} }{2+ \frac{2(f(x_{k+m+1})-f(x_{k+m} ) ) }{\alpha_{k+m} \|G(x_{k+m} ) \|_2^2 } }\right|\\\nonumber
                =&\ahat_{k,m+1}\left|1 - \frac{1}{\ahat_{k,m+1}}\cdot\frac{\frac12 s_{k+m}^{\T}s_{k+m} }{f(x_{k+m+1})-f(x_{k+m})-g_{k+m}^{\T} (x_{k+m+1}-x_{k+m} )}\right|\\\nonumber
                =& {\ahat_{k,m+1}}   |a(1-b)+b |\\\nonumber
                \leq &\frac{1}{\Lambda_2} (|a|+|b|+|ab|)\\\nonumber
                \leq &{\frac{1}{\Lambda_2}\left[\left(\frac{2 \Lambda_1 h}{\Lambda_2}+\frac{32\Lambda_1^2}{ \Lambda_2^3}\left(c_{2, m+1}+c_{2, m}+c_{3, m} c_{4, m}+\Lambda_1 c_{1, m+1}+\Lambda_1 c_{1, m}\right)\right.\right.}\\ \nonumber
                +& {\left.\left.\frac{64\Lambda_1^3 h}{\Lambda_2^4}\left(c_{2, m+1}+c_{2, m}+c_{3, m} c_{4, m}+\Lambda_1 c_{1, m+1}+\Lambda_1 c_{1, m}\right) \delta_m^2\right)\right]}\|x_k - x^*\|_2\\
                =&:{\mu \left(\Lambda_1, \Lambda_2, c_{1, m}, c_{2, m}, c_{3, m}, c_{4, m}, c_{5, m}, \delta_m \right)\|x_k - x^*\|_2.}
            %\end{split}
        \end{align}
    This ensures  \eqref{eq:C5}  for any sufficiently small $\delta$. Consequently, by taking
    \begin{equation*}
          {  \left\{
    \begin{array}{l}
c_{1, m+1}=c_{1, m}+\frac{1}{\Lambda_2} c_{3, m}+\Lambda_1 c_{5, m}, \\
c_{2, m+1}=\frac{L}{6}(m+1)^3 c_{4, m}^3+\frac{1}{2} \Lambda_1 \left(2 (m+1) c_{4, m}+c_{1, m+1} \delta_{m+1}\right) c_{1, m+1}, \\
c_{3, m+1}=\frac{L}{2}(m+1) c_{4, m}+\Lambda_1 c_{1, m + 1},\\
c_{4, m+1}=\frac{\Lambda_1}{\Lambda_2}(m+1) c_{4, m}, \\
c_{5, m+1}=\mu \left(\Lambda_1, \Lambda_2, c_{1, m}, c_{2, m}, c_{3, m}, c_{4, m}, c_{5, m}, \delta_m \right),
\end{array}\right.}
    \end{equation*}
{with $\delta_0=\rho, \delta_{m+1}=\min \left\{\frac{\delta_m}{(m+1) c_{4, m}}, \delta_m\right\} $ and $\mu(\cdot)$ a positive function depending on the parameters in the bracket,} we have  proved \eqref{C1}-\eqref{C5} by induction.
    \end{proof}

With the help of the key Lemma \ref{lm1}, the proof for \cite[Theorem 2.3]{dahs:2006} can be  analogously applied to conclude the following theorem; due to the very similarity, we omit the detailed proof.  
    \begin{theorem}\label{localcon}
    Under {{\bf Assumptions} \ref{assum:A1}  and \ref{assum:A3}}, there exists a positive number $r<\rho$, such that for any $0<\delta < r$ and $x_0,x_1 \in B_{\delta} (x^*)$, the sequence $\{x_k\}$ generated by iteration \eqref{eq:lsKGD} converges to $x^*$ with at least an R-linear rate.
    \end{theorem}

\section{Numerical experiments}{\label{sec:numer} }
To evaluate the performance of the KGD step-size strategies and our adaptive version {\tt KGDadp}, we implement  Algorithm \ref{alg:global} on MATLAB R2015b and run it on the Ubuntu 20.04 system on a 64-bit PC with an Intel Core I5 8550U CPU (3.0GHz) and 8GB of RAM. As the stopping criterion that is used in stabilized version ({\tt BBstab}) of BB   proposed in \cite{budh:2019}, we terminate the iteration whenever    
\begin{equation}{\label{end}}
         \|g_k\|_2 \leq 10^{-6} \cdot \|g_0\|_2.
    \end{equation}
For the parameters involved in Algorithm \ref{alg:global}, we set $k_{\max}=10^5, ~\alpha_0 = \frac{1}{||g_0||_2},~\eta = 10^{-4}$ and $M = 20.$  
    
{    
Two parts of test problems are  used. The first part consists general unconstrained minimization problems  from CUTEst collection\footnote{Refer to \url{https://github.com/ralna/CUTEst/wiki/} for the CUTEst environment  and  \url{https://www.cuter.rl.ac.uk/Problems/mastsif.shtml} for the CUTEst test problems.} \cite{gort:2015}; in particular, we choose unconstrained optimization problems with dimension less than or equal to 15000; this then results in a set of  $212$ problems (see Table \ref{tab:cutest} in Appendix).  Note that for each problem, the initial point  $x_0$ is also provided in the CUTEst collection. The second part of test problems are the convex minimizations from the practical logistic regression \cite{mami:2020}; both synthetic and real datasets  \cite{mami:2020} are utilized as test cases, and the performance of two adaptive gradient methods recently proposed for the convex minimization \cite{mami:2020} are also reported for comparison.}
    
    %%%%%%%%%%%%%
    \subsection{Performance of {\tt KGDadp} with different step-size strategies}\label{subsec:compfoursteps}
    We first evaluate  KDG step-size strategies and  BB steps in the framework of  Algorithm \ref{alg:global}. In particular, we compare four versions of {\tt KGDadp}, namely, 
    \begin{equation}\label{eq:fourmethods}
      {\tt KGDadp}(\Kzero,\Kone),\quad {\tt KGDadp}(\Kzero,\Kone^{\rm s}),\quad  {\tt KGDadp}(\Kzero,{\rm BB1}), \quad {\tt KGDadp}(\Kzero,{\rm BB2}).
      \end{equation}   
    To illustrate the efficiency of each tested method, we use the performance profile of Dolan and Mor{\'e} \cite{domo:2002} to demonstrate the results. In Figure \ref{Compare1}, the number of out-loop iterations $k$,  the number of computing gradients as well as the consumed CPU time (in second) are presented. To understand Figure \ref{Compare1}, for example,  let $ {\rm iter}_{p,s} $ denote the number of iterations that a solver $s$ requires in solving the problem $ p $, and define the ratio $\frac {{\rm iter}_{p,s} }{{\rm iter}^*_p}$,
where ${\rm iter}^*_p $ represents the smallest number of iterations among the four solvers for solving the problem $p$. %{\color{red}{It should be pointed out that whenever $s$ fails  in a problem, the ratio $\frac {{\rm iter}_{p,s} }{{\rm iter}^*_p}$ is set to be the maximal number $ k_{\max}=10^{5} $ for $s$.}}\marginpar{\tiny check it.} 
Let ${\cal P}$ be our set of $n_p=212$ test problems; then the distribution function  
\begin{equation*}
	r_{\rm iter}(\tau) = \frac{{\rm size}\left\{p\in {\cal P}: \log_2\left(\frac {{\rm iter}_{p,s} }{{\rm iter}^*_p}\right) \le \tau\right\}}{n_p}, \quad \tau \ge 0,
\end{equation*}
 is a performance metric for the solver $s$. The performance profiles of Figure \ref{Compare1}   plot the distribution functions $r_{\rm iter}(\tau)$, $r_{\rm grad}(\tau) $ and $r_{\rm cpu}(\tau) $ for each solver. 
 
 Besides the  performance profiles, in Table \ref{res1}, we also report the number of cases that an algorithm successfully computes a solution to meet the condition \eqref{end}. One can see from  Figure \ref{Compare1} and Table \ref{res1} that   the {short KGD step-size}  has the best performance on this set of test problems.
    \begin{figure}[htbp]
    \centering
    \hfill
    \subfloat{\includegraphics[width = 1.93in]{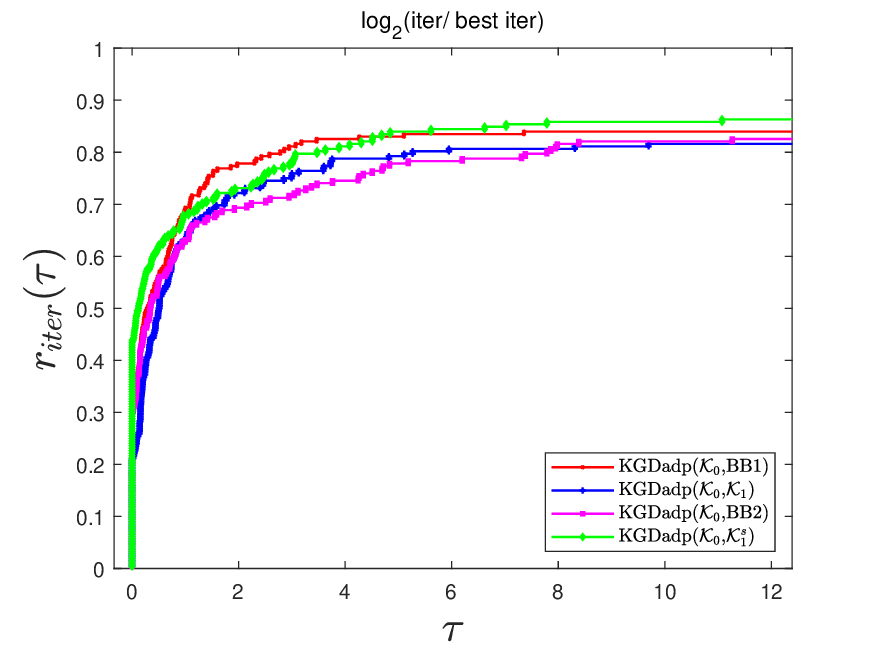} \label{iter1}}
    \hfill
    \subfloat{\includegraphics[width = 1.93in]{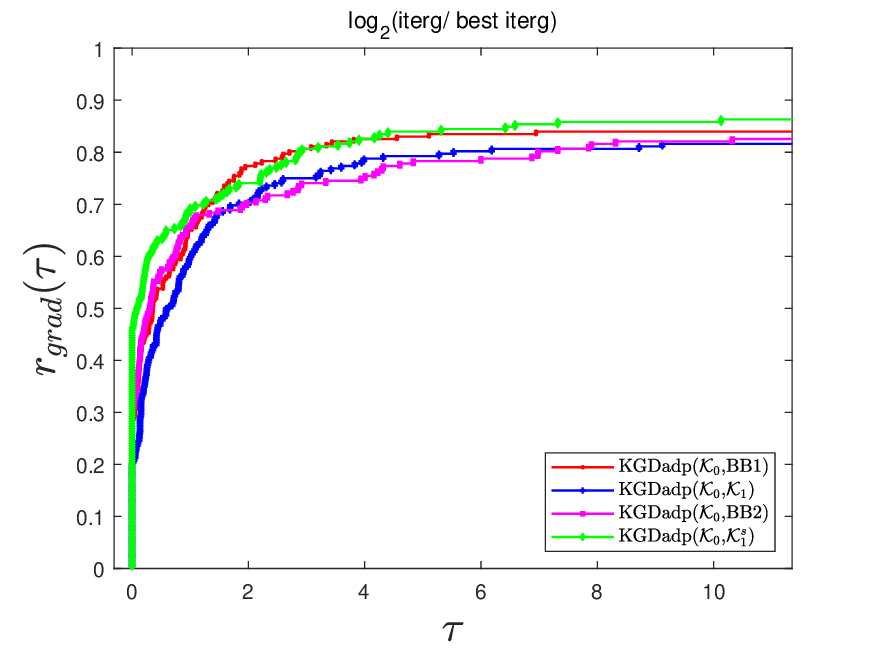} \label{iterg1} }
    \hfill
    \subfloat{\includegraphics[width = 1.93in]{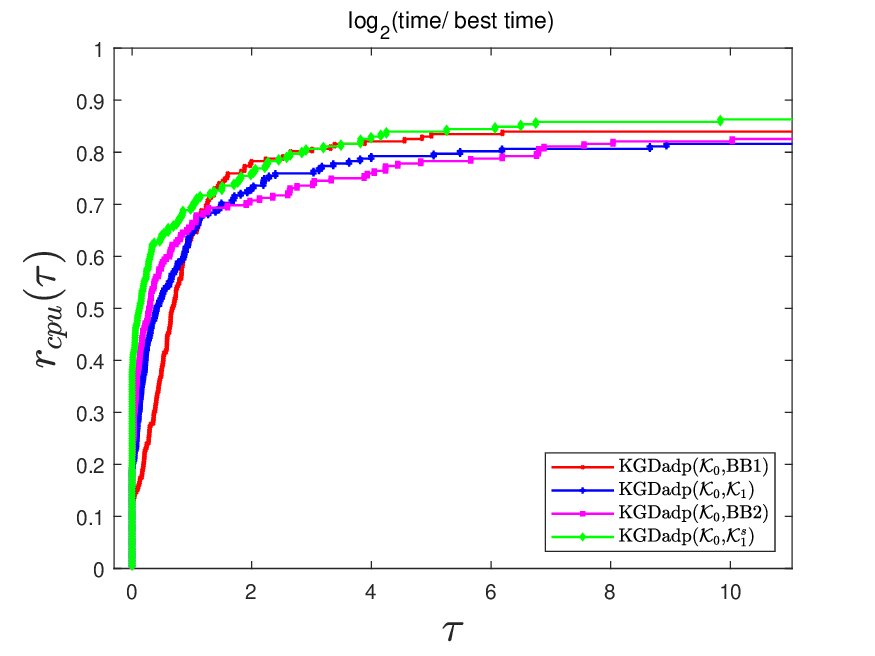} \label{CPUtime1} }
    \caption{\em Performance profiles of four implementations \eqref{eq:fourmethods} of  {\tt KGDadp} on 212 unconstrained problems from CUTEst collections: the number of iterations (left),  the number of calculations of gradient (middle) and CPU time (right).}
    \label{Compare1}
    \end{figure}

\begin{table}[htbp]
  \centering
  \setlength\tabcolsep{3pt}  % set column space
	\renewcommand\arraystretch{.8} % set row space
  \caption{Number of successful cases}
    \begin{tabular}{lcccc}
    \toprule
    \small
    Method & \multicolumn{1}{l}{ ${\tt KGDadp}(\Kzero,{\rm BB1})$} & \multicolumn{1}{l}{ ${\tt KGDadp}(\Kzero,\Kone)$} & \multicolumn{1}{l}{ ${\tt KGDadp}(\Kzero,{\rm BB2})$} & \multicolumn{1}{l}{ ${\tt KGDadp}(\Kzero,\Kone^{\rm s})$} \\
    \midrule
    Successful cases & 178   & 173   & 175   & 183 \\
    \bottomrule
    \end{tabular}%
  \label{res1}
\end{table}

%%%%%%%%%%%%%
\subsection{Performance of the long BB step in different frameworks}\label{subsec:BB1comp}
The long BB step-size \eqref{BBl} is probably the most widely-used in the line-search gradient methods, and there are also some globalization strategies to handle the divergence for the general minimization. For this purpose, a recent work  in  \cite{budh:2019}  proposes an effective stabilized version ({\tt BBstab}) of BB steps. Precisely, instead of using the pure long BB step-size \eqref{BBl},  the following  modified 
     \begin{equation}\nonumber
         \alpha_{k+1} = \min \left\{ \alpha_{k+1}^{\text{BB1}}, \frac{\Delta}{\|g_k\|_2} \right\},
     \end{equation}
     is used,  where $\Delta > 0$ is a parameter, set as \cite{budh:2019} 
     \begin{equation}\nonumber
         \Delta = c \cdot \min \left\{ \|s_1\|_2, \|s_2\|_2, \|s_3\|_2\right\},
     \end{equation}
      $s_k = x_k - x_{k-1}$ and $c > 0$ is a positive constant. Considering that our {\tt KGDadp} is also a globalization strategy for BB and KGD,  in this subsection, we shall evaluate the effectiveness of {\tt KGDadp} on the long BB step. 
     
     To this end, we compare the performance of ${\tt KGDadp}(\Kzero,{\rm BB1})$ to {\tt BB1stab} with two parameters $c = 1$ and $c = 0.5$. In  Figure \ref{Compare2}, we give the performance profiles, and also report the number of  successful cases in  Table \ref{res2}. The method labelled as `Pure BB1' refers to the iteration $x_{k+1} = x_k - \alpha_{k}^{\text{BB1}} \cdot g_k$. One can see that  ${\tt KGDadp}(\Kzero,{\rm BB1})$ performs better than  {\tt BB1stab} on our test problems.
     
    \begin{figure}[htbp!]
    \centering
      \setlength\tabcolsep{3pt}  % set column space
	\renewcommand\arraystretch{.8} % set row space
    \hfill
    \subfloat{\includegraphics[width = 1.9in]{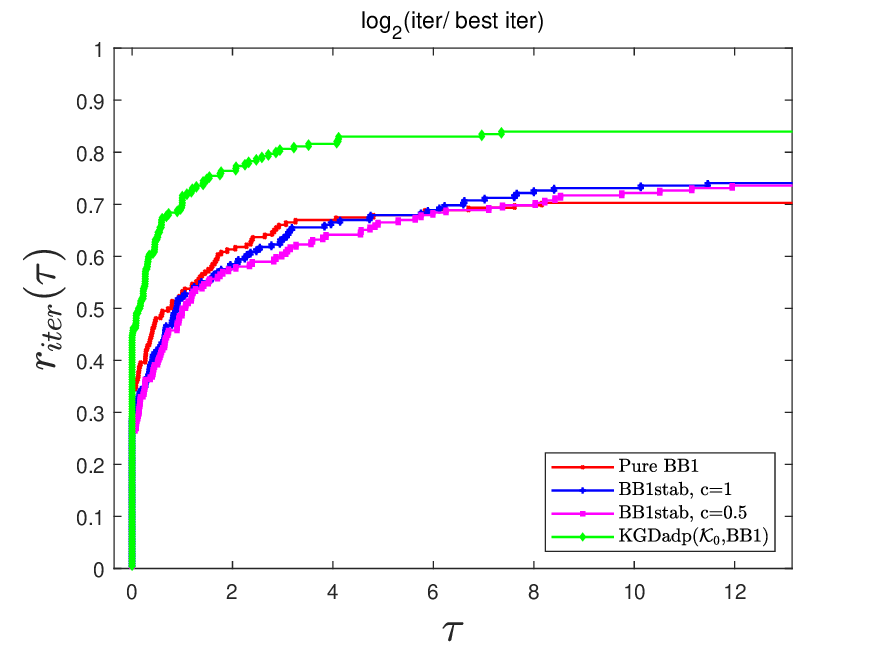} \label{iter2}}
    \hfill
    \subfloat{\includegraphics[width = 1.9in]{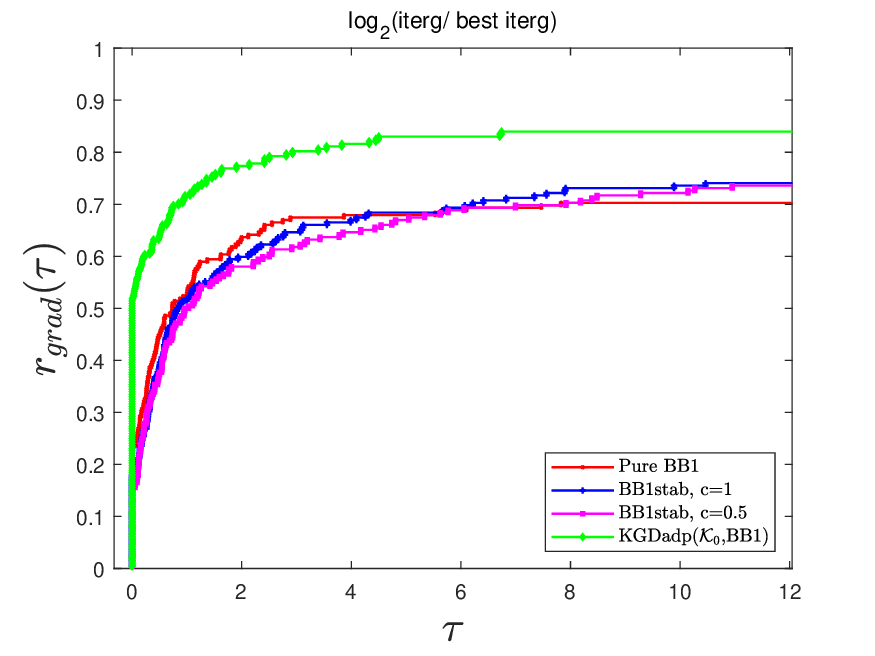} \label{iterg2} }
    \hfill
    \subfloat{\includegraphics[width = 1.9in]{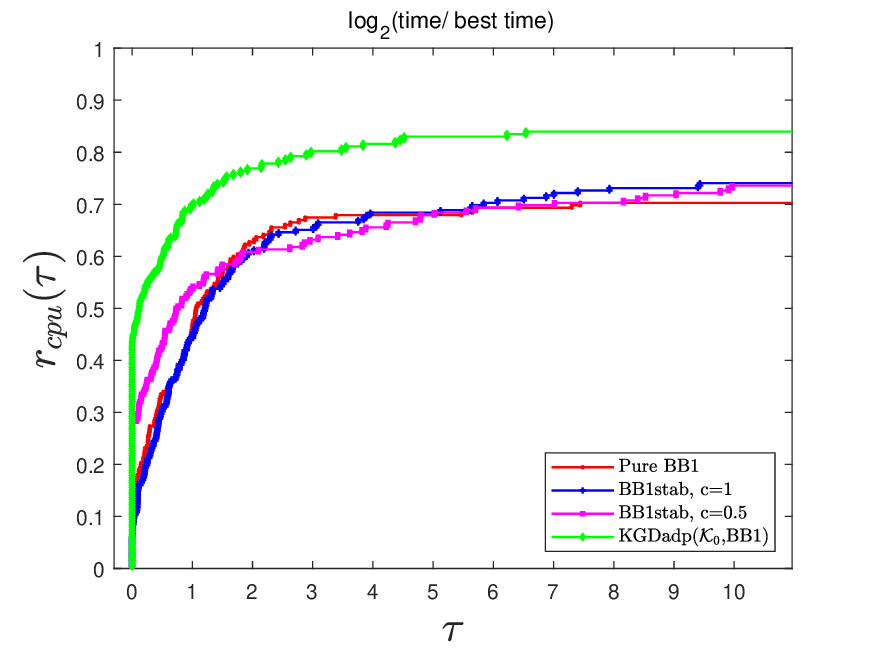} \label{CPUtime2} }
    \caption{\em Performance profiles of {\rm Pure BB1}, {\tt BB1stab} and ${\tt KGDadp}(\Kzero,{\rm BB1})$. }
    \label{Compare2}
    \end{figure}
    \begin{table}[htbp]
    \centering
    \caption{Number of successful cases}
    \begin{tabular}{lcccc}
    \toprule
    Method & \multicolumn{1}{l}{Pure BB1} & \multicolumn{1}{l}{{\tt BB1stab}($c=1$)} & \multicolumn{1}{l}{{\tt BB1stab}($c=0.5$)} & \multicolumn{1}{l}{${\tt KGDadp}(\Kzero,{\rm BB1})$ }\\
    \midrule
    Successful cases & 149   & 157   & 156   & 178 \\
    \bottomrule
    \end{tabular}%
    \label{res2}
    \end{table}

    {
   \subsection{Performance of {\tt KGDadp} on the logistic regression problem}
   As the final part of numerical test, we apply the proposed methods to solve an important practical convex problem, the logistic regression:
   \begin{equation}\label{eq:logistic}
   \min_{x\in {\mathbb R}^{n}} f(x)=-\frac{1}{m}\sum_{i=1}^m \left(y_i \log (s(a_i^{\T} x)) + (1 - y_i) \log (1 - s(a_i^{\T} x))\right) + \frac{\gamma}{2}\|x\|_2^2,
   \end{equation}
where  $a_i\in {\mathbb R}^{n}$, $y_i\in \{0, 1\}$, $s(z)=\frac{1}{1+\exp(-z)}$ is the sigmoid function. The objective function $f(x)$ in \eqref{eq:logistic} is smooth   with smoothness constant $L=\frac{1}{4}\lambda_{n}(A^{\T} A) + \gamma$ \cite{mami:2020} and its gradient is   $\nabla f(x) = \frac{1}{m}\sum_{i=1}^m a_i(s(a_i^{\T} x)-y_i) + \gamma x$, where $A=[a_1^{\T},\dots,a_m^{\T}]^{\T}$. 
   For comparison, we also employ   two  adaptive gradient-based methods, namely {\tt AdGD} and {\tt AdGD-accel},  proposed recently  for convex problems   in \cite{mami:2020}. {\tt AdGD} is a basic and simple adaptive gradient-based  method while {\tt AdGD-accel} is an acceleration version. As both methods\footnote{\url{https://github.com/ymalitsky/adaptive_GD}} are coded in  Python, we provide the Python version of our methods  ${\tt KGDadp}(\Kzero,{\rm BB1})$ and ${\tt KGDadp}(\Kzero,\Kone^{\rm s})$. For a fair comparison, the default settings of  {\tt AdGD}  and {\tt AdGD-accel} given in \cite{mami:2020} are used except for revising the stopping rules with the same one \eqref{end}. 
    }

{
To evaluate the performance of the methods  ${\tt KGDadp}(\Kzero,{\rm BB1})$,  ${\tt KGDadp}(\Kzero,\Kone^{\rm s})$, {\tt AdGD}  and {\tt AdGD-accel}, we use both the synthetic data and real data for testing. In particular, the synthetic data consist of $100$ randomly generated matrices $A\in \bbR^{m\times 500}$ as well as  the corresponding labels $\{y_i\}_{i=1}^m$ for each case with $m\in \{1000,2000,\dots,100000\}$.   The performance profile of Dolan and Mor{\'e} \cite{domo:2002}  used in Figures \ref{Compare1} and \ref{Compare2} is similarly applied to demonstrate the results for the   synthetic data; particularly, Figure \ref{Compare3} reports the number of iterations  and CPU time. The results clearly show that    ${\tt KGDadp}(\Kzero,\Kone^{\rm s})$ and ${\tt KGDadp}(\Kzero,{\rm BB1})$ are more efficient than {\tt AdGD}  and {\tt AdGD-accel}. 
    \begin{figure}[htbp]
    \centering
    \hskip2mm 
    \subfloat{\includegraphics[width = 2.55 in]{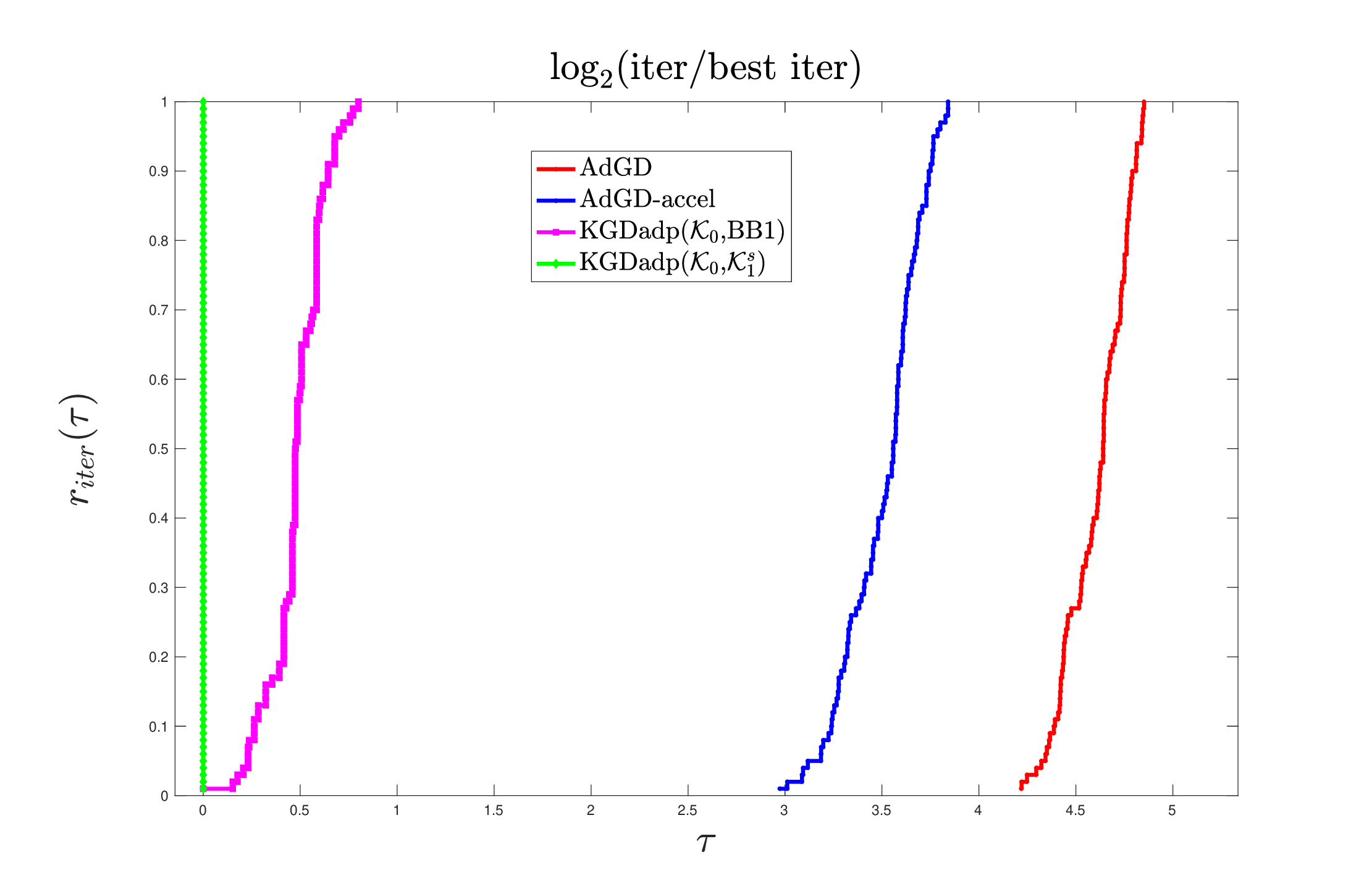} \label{logis_iter}}
    \hfill
    \subfloat{\includegraphics[width = 2.55 in]{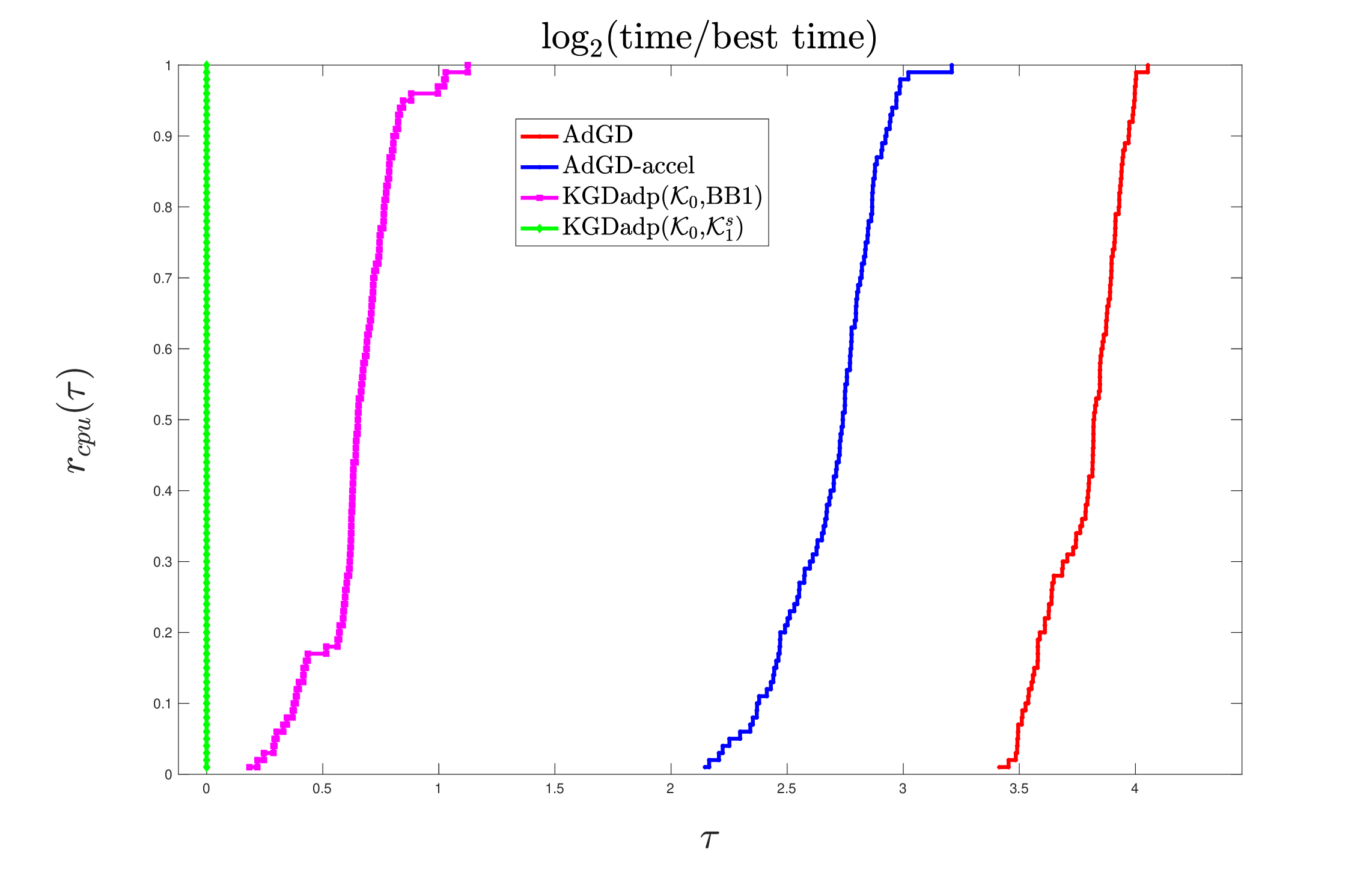} \label{logis_cpu} }
    \caption{\em  {Performance profiles of four methods   ${\tt KGDadp}(\Kzero,{\rm BB1})$,  ${\tt KGDadp}(\Kzero,\Kone^{\rm s})$, {\tt AdGD}  and {\tt AdGD-accel} on  the logistic regression \eqref{eq:logistic} with 100 randomly generated cases: the number of iterations (left)  and CPU time (right).}}
    \label{Compare3}
    \end{figure} 
}

{
For the real data, we follow \cite{mami:2020} and use  `mushrooms' ($m=8124, n=112$), `covtype'  ($m=581012, n=54$) and `wa8' ($m=49749, n=300$) datasets from the [LIBSVM] (\url{https://www.csie.ntu.edu.tw/~cjlin/libsvm/}) library to run the experiments.   The relative   $\|g_k\|_2/\|g_0\|_2$ of the gradient $g_k$ of each method is depicted versus the iteration number $k$ in Figure \ref{Compare4}, from which the faster convergence of ${\tt KGDadp}(\Kzero,{\rm BB1})$ and ${\tt KGDadp}(\Kzero,\Kone^{\rm s})$ can be clearly observed. 
\begin{figure}[htbp]
    \centering
    \hfill
    \subfloat{\includegraphics[width = 1.932in]{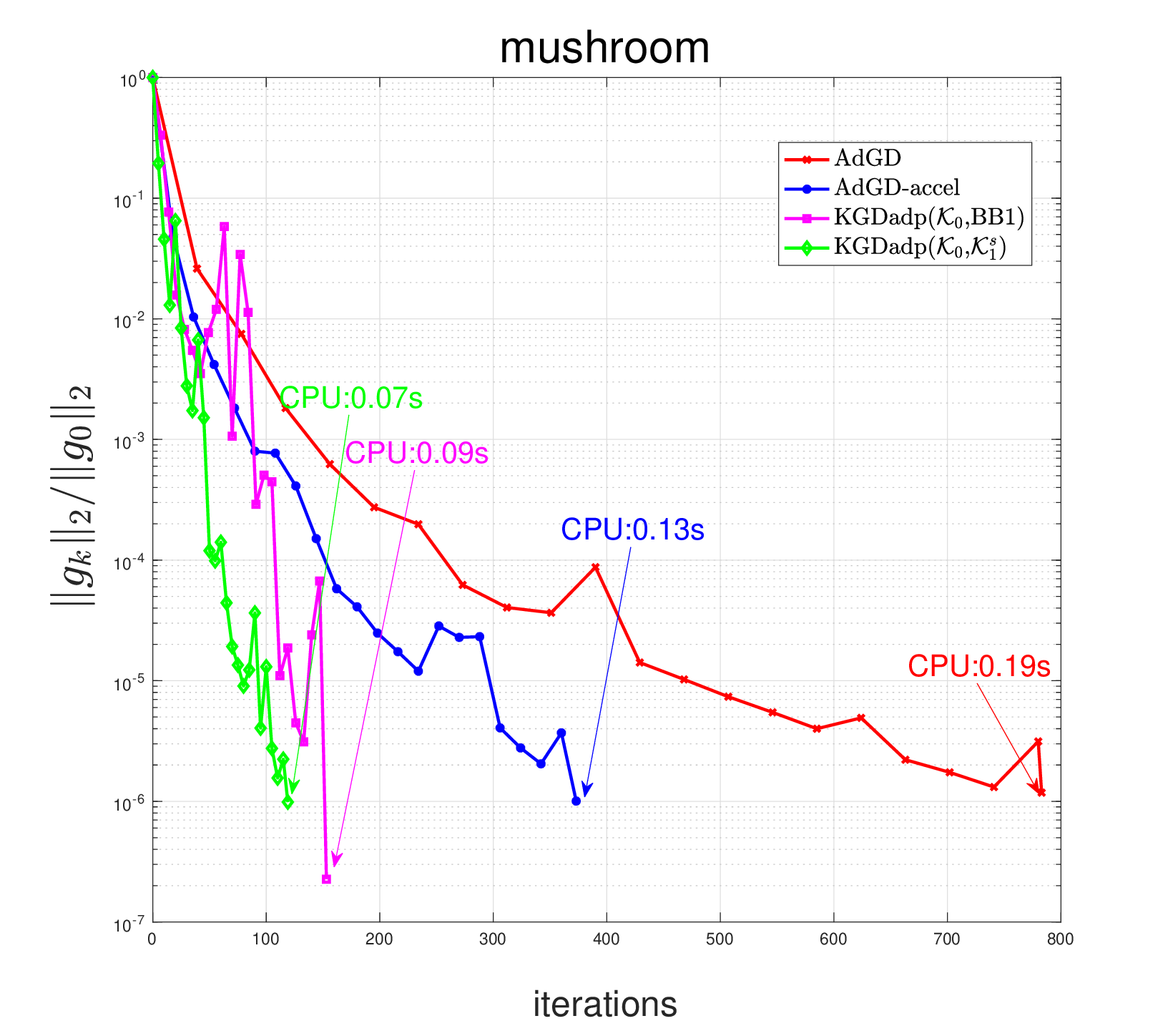} \label{mushroom}}
    \hfill
    \subfloat{\includegraphics[width = 1.932in]{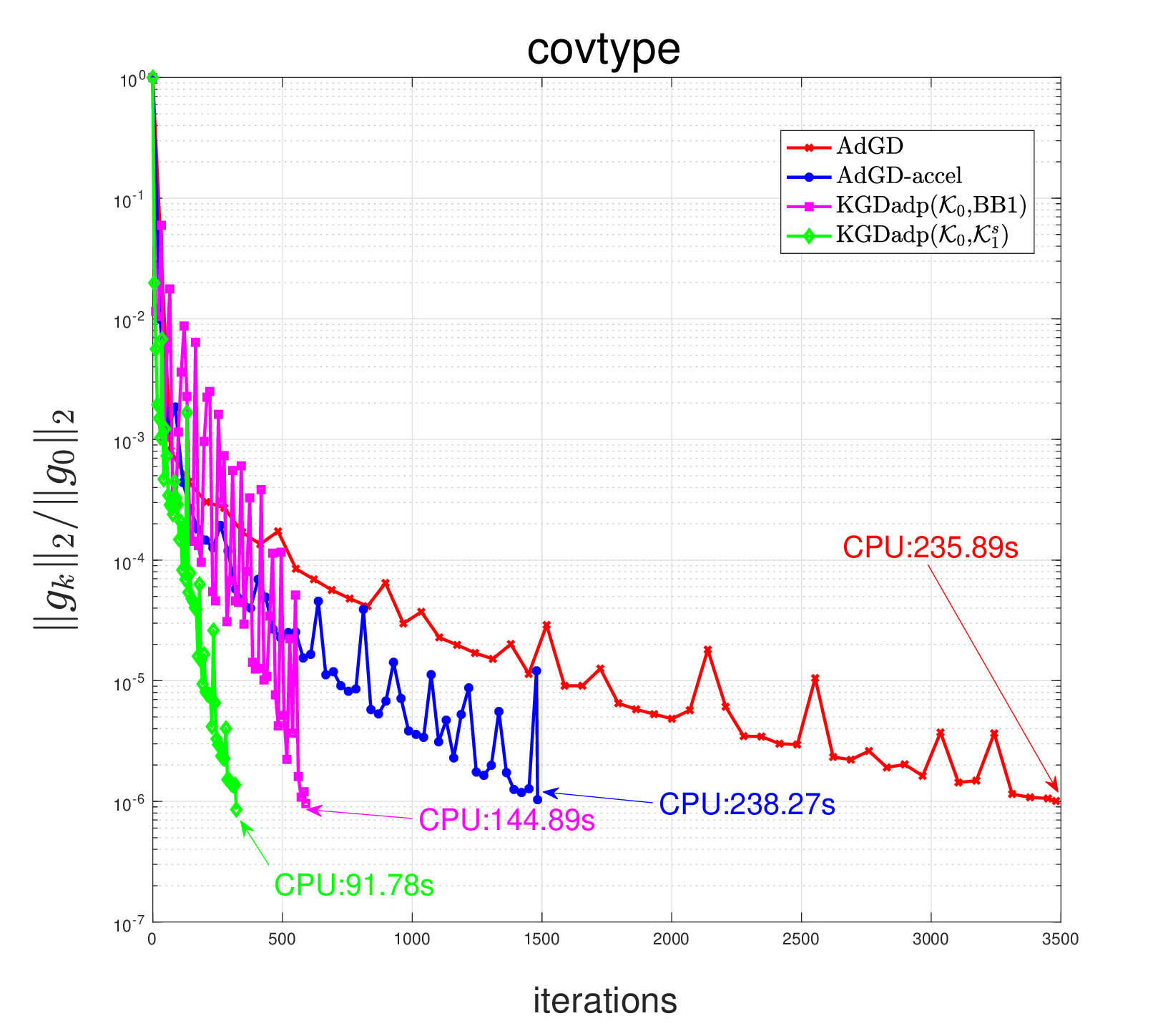} \label{covtype} }
    \hfill
    \subfloat{\includegraphics[width = 1.932in]{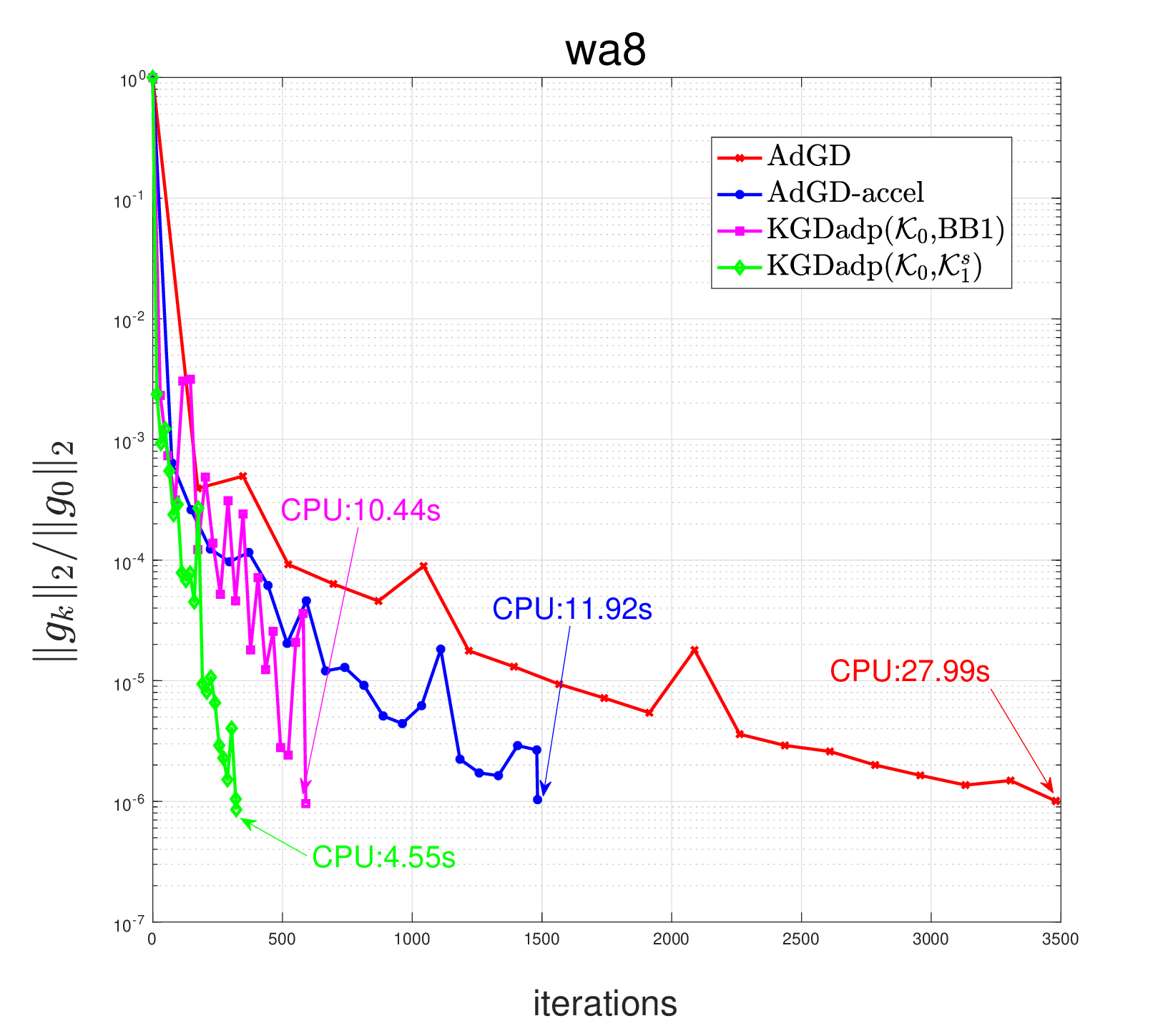} \label{wa8} }
    \caption{\em { The relative $\|g_k\|_2/\|g_0\|_2$ of the gradient $g_k$  versus the iteration number $k$ for ${\tt KGDadp}(\Kzero,{\rm BB1})$,  ${\tt KGDadp}(\Kzero,\Kone^{\rm s})$, {\tt AdGD}  and {\tt AdGD-accel} on  the logistic regression \eqref{eq:logistic} with  the `mushrooms', `covtype' and `wa8' datasets.}}
    \label{Compare4}
    \end{figure} 
}
    
\section{Concluding remarks}{\label{sec:conclu} }
    In this paper, we introduced  William Kahan's automatic step-size control strategies \cite{kahn:2019b,kahn:2019a} on the line-search gradient methods. One of our motivations is a connection given in  \cite{gjmx:2023} which establishes the equivalence between the long KGD step-size and the long BB step in the quadratic model. In order to effectively produce the step-size for the current iteration along  the negative gradient direction, both BB and KGD step-sizes attempt to use the information that is carried in the previous iteration (including the value of the objective function, its gradient as well as the step-size). Although the ways to deduce the corresponding step-size formulae are different, it is interesting to note that, at a current iteration, both essentially achieve the same step-size for the strongly  convex quadratic model, provided that they share the same previous iterate. Different from the formulae  of BB steps, the KGD step-sizes iterate directly on the previous step-size, and in this sense, KGD provides an alternative way for the BB step. 
        
     Based on the previous discussion  in \cite{gjmx:2023} for Kahan's  automatic step-size control \cite{kahn:2019b,kahn:2019a}, we have made a further step in this paper. By the connection of the long BB step with the KGD step-size for Regime 1, we  derived the short KGD step-size formulation, which reduces to the short BB step in the quadratic case; moreover, through a change of variables, we have established a correspondence between the sequences from the long and short BB steps. Thus, the recent convergence result  \cite{lisu:2021} on the long BB step can be easily applied to the short BB step (and hence the short KGD step-size) to show the R-linear convergence with a rate at least $1-\frac{1}{\kappa}$. In order to dynamically and effectively choose suitable  KGD steps for Regime 0 and Regime 1 during iteration, we further constructed an adaptive framework {\tt KGDadp} in Algorithm \ref{alg:global}, and proved its global convergence. Locally, we showed that both the long and short  KGD steps converge R-linearly. Our preliminary numerical experiments on a set of CUTEst collection \cite{gort:2015} {as well as the convex logistic regression problems \cite{mami:2020}} indicate that KGD step-sizes deserve further investigations.
                 
%    In the present paper, we apply Kahan's automatic step-size control method in gradient descent algorithm successfully. Steps available were provided in \textbf{Regime 1} and the algorithm was globalized by steps in \textbf{Regime 0}. Both theoretically and numerically, we showed the high efficiency of Kahan's global gradient descent method. The results were encouraging.
%    \par
%    Besides the classic gradient descent, Kahan mentioned some other usage on his step size control in \cite{kahn:2019b,kahn:2019a}. In \textbf{Regime 1}, gradient method with momentum, such as the heavy ball method and Nestrov's acclerated method, performs well and has gained enough attention. We refer to \cite{mech:2011,nest:1983,nest:2003,nest:2018}. It was pointed out that, with a proper manner applied, Kahan's automatic step-size can be also adapted in these models to gain a high efficiency in his numerical experiments. We hope that a better algorithm can be found in this respect.

\section*{Acknowledgements}
 {The authors would like to thank the anonymous referees for their careful reading, useful comments and suggestions to improve the presentation of the paper.} Also, they  thank Prof. Na Huang at China Agricultural University for providing the code of the stabilized Barzilai-Borwein method ({\tt BB1stab}) \cite{budh:2019}  for our numerical testing, and  thank     
    Prof. Ren-Cang Li at the University of Texas at Arlington for sharing the re-typeset latex form of \cite{kahn:2019b} and also helpful discussions on  Kahan's automatic step-size control strategies.
    
%    
%   
%\section*{Ethical Approval} Not Applicable.
%
%
%
%
%\section*{Availability of supporting data} 
%The datasets generated during and/or analysed during the current study are available in {\url{https://www.cuter.rl.ac.uk/Problems/mastsif.shtml} and \url{https://www.csie.ntu.edu.tw/~cjlin/libsvm/}}.
%
%\section*{Conflict of interest}  The authors declare that they have no conflict of interest.
%
%
%
%\section*{Funding}
% The work of Lei-Hong Zhang was  supported in part by the National Natural Science Foundation of China
%        NSFC-12071332, NSFC-12371380,   Jiangsu Shuangchuang Project (JSSCTD202209), and Academic Degree and Postgraduate Education Reform Project of Jiangsu Province. 
%        
%       
%
%\section*{Author Contributions}  L-HZ and YM conceived the framework. YM and LX contributed to theoretical analysis. CS, YM  and LX contributed to coding and numerical testing.   L-H Z and YM wrote the manuscript. 
%
%All authors contribute equally.

%{\bf Conflict of interest} The authors declare that they have no conflict of interest.
 
 \def\noopsort#1{}\def\l{\char32l}\def\v#1{{\accent20 #1}}
  \let\^^_=\v\def\hbk{hardback}\def\pbk{paperback}
\providecommand{\href}[2]{#2}
\providecommand{\arxiv}[1]{\href{http://arxiv.org/abs/#1}{arXiv:#1}}
\providecommand{\url}[1]{\texttt{#1}}
\providecommand{\urlprefix}{URL }

%    
%{\small
%\bibliographystyle{aims}
%\bibliography{/Users/longzlh/360Icloud/References/Latex/Bib/strings,/Users/longzlh/360Icloud/References/Latex/Bib/zhang-li}

\begin{thebibliography}{10}
\small
\bibitem{babo:1988}
\newblock J.~Barzilai and J.~M. Borwein,
\newblock Two-point step size gradient methods,
\newblock \emph{IMA J. Numer. Anal.}, \textbf{8},
\newblock \urlprefix\url{https://doi.org/10.1093/imanum/8.1.141}.

\bibitem{budh:2019}
\newblock O.~Burdakov, Y.~H. Dai and N.~Huang,
\newblock Stabilized {B}arzilai-{B}orwein method,
\newblock \emph{J. Comput. Math.}, \textbf{37} (2019), 916--936.

\bibitem{dafl:2005}
\newblock Y.~H. Dai and R.~Fletcher,
\newblock On the asymptotic behaviour of some new gradient methods,
\newblock \emph{Math. Program., Ser. A}, \textbf{103} (2005), 541--559.

\bibitem{dahs:2006}
\newblock Y.~H. Dai, W.~W. Hager, K.~Schittkowski and H.~C. Zhang,
\newblock The cyclic {Barzilai-Borwein} method for unconstrained optimization,
\newblock \emph{IMA J. Numer. Anal.}, \textbf{26} (2006), 604--627.

\bibitem{dali:2002}
\newblock Y.~H. Dai and L.~Z. Liao,
\newblock {R}-linear convergence of the {B}arzilai and {B}orwein gradient
  method,
\newblock \emph{IMA J. Numer. Anal.}, \textbf{22} (2002), 1--10.

\bibitem{domo:2002}
\newblock E.~D. Dolan and J.~Mor\'e,
\newblock Benchmarking optimization software with performance profiles,
\newblock \emph{Math. Program.}, \textbf{91} (2002), 201--213.

\bibitem{flet:2005}
\newblock R.~Fletcher,
\newblock Optimization and control with applications,
\newblock Springer, Boston, MA, 2005,
\newblock chapter On the {Barzilai-Borwein} method, 235--256.

\bibitem{gort:2015}
\newblock N.~I.~M. Gould, D.~Orban and P.~L. Toint,
\newblock {CUTEst}: a constrained and unconstrained testing environment with
  safe threads for mathematical optimization,
\newblock \emph{Comput. Optim. Appl.}, \textbf{60} (2015), 545--557.

\bibitem{grip:1986}
\newblock L.~Grippo, F.~Lampariello and S.~Lucidi,
\newblock A nonmonotone line search technique for {N}ewton's method,
\newblock \emph{SIAM J. Numer. Anal.}, \textbf{23} (1986), 707--716.

\bibitem{gjmx:2023}
\newblock Z.~Gu, K.~Jin, Y.~Meng, L.~Xue and L.-H. Zhang,
\newblock On {K}ahan's automatic step-size control and an anadromic gradient
  descent iteration,
\newblock \emph{Optimization},
\newblock   \textbf{74}:1 (2025), 189-218.

\bibitem{kahn:2019b}
\newblock W.~Kahan,
\newblock \emph{Automatic Step-Size Control for Minimization Iterations},
\newblock Technical report, University of California, Berkeley CA, USA, 2019,
\newblock
  \urlprefix\url{https://people.eecs.berkeley.edu/~wkahan/26Sept19.pdf}.

\bibitem{kahn:2019a}
\newblock W.~Kahan,
\newblock \emph{A Numerical Analyst thinks about ``Deep Learning" and
  Artificial Intelligence},
\newblock Technical report, University of California, Berkeley CA, USA, 2019,
\newblock \urlprefix\url{https://people.eecs.berkeley.edu/~wkahan/7Nov18.pdf}.

\bibitem{lats:2023}
\newblock P.~Latafat, A.~Themelis, L.~Stella and P.~Patrinos,
\newblock \emph{Adaptive proximal algorithms for convex optimization under
  local {L}ipschitz continuity of the gradient},
\newblock Technical report, arXiv:2301.04431, 2023.

\bibitem{lisu:2021}
\newblock D.~Li and R.~Sun,
\newblock \emph{On a Faster {R}-Linear Convergence Rate of the
  {B}arzilai-{B}orwein Method},
\newblock arxiv:2101.00205, 2021.

\bibitem{mami:2020}
\newblock Y.~Malitsky and K.~Mishchenko,
\newblock Adaptive gradient descent without descent,
\newblock in \emph{{Proceedings of the 37th International Conference on Machine
  Learning}}, 2020,
\newblock 6702--6712.

\bibitem{nowr:2006}
\newblock J.~Nocedal and S.~Wright,
\newblock \emph{Numerical Optimization},
\newblock 2nd edition,
\newblock Springer, New York, 2006.

\bibitem{rayd:1993}
\newblock M.~Raydan,
\newblock On the {Barzilai and Borwein} choice of steplength for the gradient
  method,
\newblock \emph{IMA J. Numer. Anal.}, \textbf{13} (1993), 321--326.

\bibitem{rayd:1997}
\newblock M.~Raydan,
\newblock The {Barzilai and Borwein} gradient method for the large scale
  unconstrained minimization problem,
\newblock \emph{SIAM J. Optim.}, \textbf{7} (1997), 26--33.

\end{thebibliography}
%}

 \newpage
\section*{Appendix}
\begin{table}[h!!!]
	\caption{\label{tab:cutest} Information on test problems selected from the \texttt{CUTEst} collection }
	\centering
		\setlength\tabcolsep{3pt}  % set column space
	\renewcommand\arraystretch{.71} % set row space
	\tabcolsep 0.13in\small
	\begin{tabular}{|c|c||c|c||c|c||c|c|}%
	\hline
      \small Problem & \small $n$ & \small Problem &  \small $n$ & \small Problem & \small  $n$&\small Problem &  \small$n$\\
	\hline
 \tiny AKIVA     	&\tiny	2	&\tiny	DIXMAANJ  	&\tiny	3000	&\tiny	HILBERTB  	&\tiny	10	&\tiny	PALMER7C  	&\tiny	8	\\ \hline \tiny
ALLINITU  	&\tiny	4	&\tiny	DIXMAANK  	&\tiny	3000	&\tiny	HIMMELBB  	&\tiny	2	&\tiny	PALMER8C  	&\tiny	8	\\ \hline \tiny
ARGLINA   	&\tiny	200	&\tiny	DIXMAANL  	&\tiny	3000	&\tiny	HIMMELBF  	&\tiny	4	&\tiny	PARKCH    	&\tiny	15	\\ \hline \tiny
ARGLINB   	&\tiny	200	&\tiny	DIXMAANM  	&\tiny	3000	&\tiny	HIMMELBG  	&\tiny	2	&\tiny	PENALTY1  	&\tiny	1000	\\ \hline \tiny
ARGLINC   	&\tiny	200	&\tiny	DIXMAANN  	&\tiny	3000	&\tiny	HIMMELBH  	&\tiny	2	&\tiny	PENALTY2  	&\tiny	200	\\ \hline \tiny
ARGTRIGLS 	&\tiny	10	&\tiny	DIXMAANO  	&\tiny	3000	&\tiny	HUMPS     	&\tiny	2	&\tiny	PENALTY3  	&\tiny	200	\\ \hline \tiny
ARWHEAD   	&\tiny	5000	&\tiny	DIXMAANP  	&\tiny	3000	&\tiny	HYDC20LS  	&\tiny	99	&\tiny	POWELLBSLS	&\tiny	2	\\ \hline \tiny
BARD      	&\tiny	3	&\tiny	DIXON3DQ  	&\tiny	10000	&\tiny	INDEF     	&\tiny	5000	&\tiny	POWELLSG  	&\tiny	5000	\\ \hline \tiny
BDQRTIC   	&\tiny	5000	&\tiny	DJTL      	&\tiny	2	&\tiny	INDEFM    	&\tiny	100000	&\tiny	POWER     	&\tiny	10000	\\ \hline \tiny
BEALE     	&\tiny	2	&\tiny	DMN15103LS	&\tiny	99	&\tiny	INTEQNELS 	&\tiny	12	&\tiny	QUARTC    	&\tiny	5000	\\ \hline \tiny
BENNETT5LS	&\tiny	3	&\tiny	DMN15332LS	&\tiny	66	&\tiny	JENSMP    	&\tiny	2	&\tiny	RAT42LS   	&\tiny	3	\\ \hline \tiny
BIGGS6    	&\tiny	6	&\tiny	DMN15333LS	&\tiny	99	&\tiny	JIMACK    	&\tiny	3549	&\tiny	RAT43LS   	&\tiny	4	\\ \hline \tiny
BOX       	&\tiny	10000	&\tiny	DMN37142LS	&\tiny	66	&\tiny	KIRBY2LS  	&\tiny	5	&\tiny	ROSENBR   	&\tiny	2	\\ \hline \tiny
BOX3      	&\tiny	3	&\tiny	DMN37143LS	&\tiny	99	&\tiny	KOWOSB    	&\tiny	4	&\tiny	ROSZMAN1LS	&\tiny	4	\\ \hline \tiny
BOXBODLS  	&\tiny	2	&\tiny	DQDRTIC   	&\tiny	5000	&\tiny	LANCZOS1LS	&\tiny	6	&\tiny	S308      	&\tiny	2	\\ \hline \tiny
BOXPOWER  	&\tiny	20000	&\tiny	DQRTIC    	&\tiny	5000	&\tiny	LANCZOS2LS	&\tiny	6	&\tiny	SBRYBND   	&\tiny	5000	\\ \hline \tiny
BRKMCC    	&\tiny	2	&\tiny	ECKERLE4LS	&\tiny	3	&\tiny	LANCZOS3LS	&\tiny	6	&\tiny	SCHMVETT  	&\tiny	5000	\\ \hline \tiny
BROWNAL   	&\tiny	200	&\tiny	EDENSCH   	&\tiny	2000	&\tiny	LIARWHD   	&\tiny	5000	&\tiny	SCOSINE   	&\tiny	5000	\\ \hline \tiny
BROWNBS   	&\tiny	2	&\tiny	EG2       	&\tiny	1000	&\tiny	LOGHAIRY  	&\tiny	2	&\tiny	SCURLY10  	&\tiny	10000	\\ \hline \tiny
BROWNDEN  	&\tiny	4	&\tiny	EIGENALS  	&\tiny	2550	&\tiny	MANCINO   	&\tiny	100	&\tiny	SCURLY20  	&\tiny	10000	\\ \hline \tiny
BROYDN3DLS	&\tiny	10	&\tiny	EIGENBLS  	&\tiny	2550	&\tiny	MARATOSB  	&\tiny	2	&\tiny	SCURLY30  	&\tiny	10000	\\ \hline \tiny
BROYDN7D  	&\tiny	5000	&\tiny	EIGENCLS  	&\tiny	2652	&\tiny	MEXHAT    	&\tiny	2	&\tiny	SENSORS   	&\tiny	100	\\ \hline \tiny
BROYDNBDLS	&\tiny	10	&\tiny	ENGVAL1   	&\tiny	5000	&\tiny	MEYER3    	&\tiny	3	&\tiny	SINEVAL   	&\tiny	2	\\ \hline \tiny
BRYBND    	&\tiny	5000	&\tiny	ENGVAL2   	&\tiny	3	&\tiny	MGH09LS   	&\tiny	4	&\tiny	SINQUAD   	&\tiny	5000	\\ \hline \tiny
CHAINWOO  	&\tiny	4000	&\tiny	ENSOLS    	&\tiny	9	&\tiny	MGH10LS   	&\tiny	3	&\tiny	SISSER    	&\tiny	2	\\ \hline \tiny
CHNROSNB  	&\tiny	50	&\tiny	ERRINROS  	&\tiny	50	&\tiny	MGH17LS   	&\tiny	5	&\tiny	SNAIL     	&\tiny	2	\\ \hline \tiny
CHNRSNBM  	&\tiny	50	&\tiny	ERRINRSM  	&\tiny	50	&\tiny	MISRA1ALS 	&\tiny	2	&\tiny	SPARSINE  	&\tiny	5000	\\ \hline \tiny
CHWIRUT1LS	&\tiny	3	&\tiny	EXPFIT    	&\tiny	2	&\tiny	MISRA1BLS 	&\tiny	2	&\tiny	SPARSQUR  	&\tiny	10000	\\ \hline \tiny
CHWIRUT2LS	&\tiny	3	&\tiny	EXTROSNB  	&\tiny	1000	&\tiny	MISRA1CLS 	&\tiny	2	&\tiny	SPMSRTLS  	&\tiny	4999	\\ \hline \tiny
CLIFF     	&\tiny	2	&\tiny	FLETBV3M  	&\tiny	5000	&\tiny	MISRA1DLS 	&\tiny	2	&\tiny	SROSENBR  	&\tiny	5000	\\ \hline \tiny
COSINE    	&\tiny	10000	&\tiny	FLETCBV2  	&\tiny	5000	&\tiny	MODBEALE  	&\tiny	20000	&\tiny	SSBRYBND  	&\tiny	5000	\\ \hline \tiny
CRAGGLVY  	&\tiny	5000	&\tiny	FLETCBV3  	&\tiny	5000	&\tiny	MOREBV    	&\tiny	5000	&\tiny	SSCOSINE  	&\tiny	5000	\\ \hline \tiny
CUBE      	&\tiny	2	&\tiny	FLETCHBV  	&\tiny	5000	&\tiny	MSQRTALS  	&\tiny	1024	&\tiny	STRATEC   	&\tiny	10	\\ \hline \tiny
CURLY10   	&\tiny	10000	&\tiny	FLETCHCR  	&\tiny	1000	&\tiny	MSQRTBLS  	&\tiny	1024	&\tiny	TESTQUAD  	&\tiny	5000	\\ \hline \tiny
CURLY20   	&\tiny	10000	&\tiny	FMINSRF2  	&\tiny	5625	&\tiny	NCB20     	&\tiny	5010	&\tiny	THURBERLS 	&\tiny	7	\\ \hline \tiny
CURLY30   	&\tiny	10000	&\tiny	FMINSURF  	&\tiny	5625	&\tiny	NCB20B    	&\tiny	5000	&\tiny	TOINTGOR  	&\tiny	50	\\ \hline \tiny
DANWOODLS 	&\tiny	2	&\tiny	FREUROTH  	&\tiny	5000	&\tiny	NELSONLS  	&\tiny	3	&\tiny	TOINTGSS  	&\tiny	5000	\\ \hline \tiny
DECONVU   	&\tiny	63	&\tiny	GAUSS1LS  	&\tiny	8	&\tiny	NONCVXU2  	&\tiny	5000	&\tiny	TOINTPSP  	&\tiny	50	\\ \hline \tiny
DENSCHNA  	&\tiny	2	&\tiny	GAUSS2LS  	&\tiny	8	&\tiny	NONCVXUN  	&\tiny	5000	&\tiny	TOINTQOR  	&\tiny	50	\\ \hline \tiny
DENSCHNB  	&\tiny	2	&\tiny	GENHUMPS  	&\tiny	5000	&\tiny	NONDIA    	&\tiny	5000	&\tiny	TQUARTIC  	&\tiny	5000	\\ \hline \tiny
DENSCHNC  	&\tiny	2	&\tiny	GENROSE   	&\tiny	500	&\tiny	NONDQUAR  	&\tiny	5000	&\tiny	TRIDIA    	&\tiny	5000	\\ \hline \tiny
DENSCHND  	&\tiny	3	&\tiny	GROWTHLS  	&\tiny	3	&\tiny	NONMSQRT  	&\tiny	4900	&\tiny	VARDIM    	&\tiny	200	\\ \hline \tiny
DENSCHNE  	&\tiny	3	&\tiny	GULF      	&\tiny	3	&\tiny	OSBORNEA  	&\tiny	5	&\tiny	VAREIGVL  	&\tiny	50	\\ \hline \tiny
DENSCHNF  	&\tiny	2	&\tiny	HAHN1LS   	&\tiny	7	&\tiny	OSBORNEB  	&\tiny	11	&\tiny	VESUVIALS 	&\tiny	8	\\ \hline \tiny
DIXMAANA  	&\tiny	3000	&\tiny	HAIRY     	&\tiny	2	&\tiny	OSCIGRAD  	&\tiny	100000	&\tiny	VESUVIOLS 	&\tiny	8	\\ \hline \tiny
DIXMAANB  	&\tiny	3000	&\tiny	HATFLDD   	&\tiny	3	&\tiny	OSCIPATH  	&\tiny	10	&\tiny	VESUVIOULS	&\tiny	8	\\ \hline \tiny
DIXMAANC  	&\tiny	3000	&\tiny	HATFLDE   	&\tiny	3	&\tiny	PALMER1C  	&\tiny	8	&\tiny	VIBRBEAM  	&\tiny	8	\\ \hline \tiny
DIXMAAND  	&\tiny	3000	&\tiny	HATFLDFL  	&\tiny	3	&\tiny	PALMER1D  	&\tiny	7	&\tiny	WATSON    	&\tiny	12	\\ \hline \tiny
DIXMAANE  	&\tiny	3000	&\tiny	HEART6LS  	&\tiny	6	&\tiny	PALMER2C  	&\tiny	8	&\tiny	WOODS     	&\tiny	4000	\\ \hline \tiny
DIXMAANF  	&\tiny	3000	&\tiny	HEART8LS  	&\tiny	8	&\tiny	PALMER3C  	&\tiny	8	&\tiny	YATP1LS   	&\tiny	2600	\\ \hline \tiny
DIXMAANG  	&\tiny	3000	&\tiny	HELIX     	&\tiny	3	&\tiny	PALMER4C  	&\tiny	8	&\tiny	YATP2LS   	&\tiny	2600	\\ \hline \tiny
DIXMAANH  	&\tiny	3000	&\tiny	HIELOW    	&\tiny	3	&\tiny	PALMER5C  	&\tiny	6	&\tiny	YFITU     	&\tiny	3	\\ \hline \tiny
DIXMAANI  	&\tiny	3000	&\tiny	HILBERTA  	&\tiny	2	&\tiny	PALMER6C  	&\tiny	8	&\tiny	ZANGWIL2  	&\tiny	2	\\ \hline  
	\end{tabular}
\end{table}
        
\end{document}